\documentclass[11pt]{article}

\usepackage{amsmath}
\usepackage{amsfonts}
\usepackage{graphicx}
\usepackage{setspace}
\usepackage{amsmath}
\usepackage{amssymb}
\usepackage{latexsym}
\usepackage{amsmath,amsfonts,amssymb,amsthm,braket,euscript,makeidx,color,mathrsfs}
\usepackage[title]{appendix}

\usepackage[numbers,sort&compress]{natbib}

\def\sqr#1#2{{\vcenter{\vbox{\hrule height.#2pt
\hbox{\vrule width.#2pt height#1pt \kern#1pt \vrule width.#2pt}
\hrule height.#2pt}}}}
\def\3n{\negthinspace \negthinspace \negthinspace }
\def\2n{\negthinspace \negthinspace }
\def\1n{\negthinspace }

   \def\cA{{\cal A}}  
     
     \def\Bc{{\bf c}}
     
\def\dbE{\mathbb{E}}

     \def\Bx{{\bf x}}
     \def\By{{\bf y}}
     \def\Bz{{\bf z}}

\def\ds{\displaystyle}

\def\={\buildrel \triangle \over =}

\def\q{\quad}

\def\max{\mathop{\rm max}}
\def\min{\mathop{\rm min}}

\def\wt{\widetilde}
\def\cd{\cdot}

\def\deq{\mathop{\buildrel\D\over=}}
\def\Re{{\mathop{\rm Re}\,}}

\usepackage[colorlinks,linkcolor=red,anchorcolor=gray,citecolor=red,urlcolor=blue]{hyperref}
\usepackage{caption}

\RequirePackage[capitalize,nameinlink]{cleveref}

\crefname{section}{section}{sections}
\crefname{subsection}{subsection}{subsections}
\Crefname{section}{Section}{Sections}
\Crefname{subsection}{Subsection}{Subsections}

\crefname{condition}{Condition}{Conditions}

\Crefname{figure}{Figure}{Figures}

\crefformat{equation}{\textup{#2(#1)#3}}
\crefrangeformat{equation}{\textup{#3(#1)#4--#5(#2)#6}}
\crefmultiformat{equation}{\textup{#2(#1)#3}}{ and \textup{#2(#1)#3}}
{, \textup{#2(#1)#3}}{ and \textup{#2(#1)#3}}
\crefrangemultiformat{equation}{\textup{#3(#1)#4--#5(#2)#6}}%
{ and \textup{#3(#1)#4--#5(#2)#6}}{, \textup{#3(#1)#4--#5(#2)#6}}{ and \textup{#3(#1)#4--#5(#2)#6}}

\Crefformat{equation}{#2Equation~\textup{(#1)}#3}
\Crefrangeformat{equation}{Equations~\textup{#3(#1)#4--#5(#2)#6}}
\Crefmultiformat{equation}{Equations~\textup{#2(#1)#3}}{ and \textup{#2(#1)#3}}
{, \textup{#2(#1)#3}}{ and \textup{#2(#1)#3}}
\Crefrangemultiformat{equation}{Equations~\textup{#3(#1)#4--#5(#2)#6}}%
{ and \textup{#3(#1)#4--#5(#2)#6}}{, \textup{#3(#1)#4--#5(#2)#6}}{ and \textup{#3(#1)#4--#5(#2)#6}}

\crefdefaultlabelformat{#2\textup{#1}#3}

\newtheorem {theorem}{Theorem}[section]
\newtheorem {lemma}{{\bf Lemma}}[section]
\newtheorem {proposition}{{\bf Proposition}}[section]
\theoremstyle{remark}
\newtheorem {remark}{{\bf Remark}}[section]

\newtheorem {condition}{{\bf Condition}}[section]
\theoremstyle{definition}
\newtheorem {definition}{{\bf Definition}}[section]
\newtheorem {example}{{\bf Example}}[section]
\theoremstyle{plain} \numberwithin {equation}{section}

\numberwithin{assumption}{section}

\allowdisplaybreaks
\everymath{\displaystyle}

\usepackage{graphicx}
\usepackage{ifpdf}
\ifpdf
\usepackage{epstopdf}
\fi

\usepackage{enumerate}

\def\deq{\mathop{\buildrel\Delta\over=}}

\begin{document}

\title{\bf Exact Controllability for Stochastic First-Order Multi-Dimensional Hyperbolic
Systems\thanks{This work is  supported by the NSF of China under grants 12025105 and 12401589. }}
\author{
Zengyu Li \thanks{School of Mathematics, Sichuan
University, Chengdu, 610064, China.  {\small\it
E-mail:} {\small\tt lizengyu@stu.scu.edu.cn}},  \quad
Qi L\"{u}\thanks{School of Mathematics, Sichuan University, Chengdu, 610064, China. {\small\it E-mail:} {\small\tt lu@scu.edu.cn}. },  \quad
Yu Wang\thanks{School of Mathematics, Southwest Jiaotong University, Chengdu, 611756, China. {\small\it E-mail:} {\small\tt yuwangmath@163.com.}} \quad
and \quad
Haitian Yang\thanks{ Department of Mathematical Sciences, Tsinghua University, Beijing, 100084, China.  {\small\it E-mail:} {\small\tt yht21@mails.tsinghua.edu.cn.}}
}

\date{\today}

\maketitle

\begin{abstract}
This paper investigates the exact controllability problem for multi-dimensional stochastic first-order symmetric hyperbolic systems  with control inputs acting in two distinct ways: an  internal control  applied to the diffusion term and a  boundary control applied to the drift term. By means of a classical duality argument, the controllability problem is reduced to an observability estimate for the corresponding backward stochastic system. The main technical contribution is the establishment of a new global Carleman estimate for such backward systems, combined with a weighted energy identity. This enables us to prove the desired observability inequality under a geometric structural condition (Condition \ref{cond1}), which ensures that all characteristic rays propagate toward the boundary within a finite time. As a result, we obtain exact controllability provided the control time $T$ exceeds a sharp threshold $T_0$ given explicitly in terms of the system geometry. Furthermore, we complement the positive result with several negative controllability theorems, which demonstrate that both controls are necessary and must act in a distributed manner. Our analysis not only extends controllability theory from deterministic to stochastic multi-dimensional hyperbolic systems but also provides, as a byproduct, new results for deterministic systems under a structural hypothesis. Applications to stochastic traffic flow, epidemiological models, and shallow-water equations are discussed.
\end{abstract}


\bigskip


\section{Introduction and main results} \label{S1}

First-order hyperbolic systems govern the evolution of numerous physical quantities. Representative examples include the Saint--Venant equations for open channels, the Aw--Rascle equations for road traffic, gas dynamics, supply chains, and heat exchangers \cite{Bartecki16, BC16, Li12}. Due to their wide range of applications, the control theory of such systems has attracted considerable attention from both mathematical and engineering communities. Extensive research has been devoted to the boundary controllability of deterministic first-order hyperbolic systems (see, e.g., \cite{Coron07, L10, Li16, Russell78} and the references therein). It is noteworthy, however, that existing controllability results are largely confined to one-dimensional (1-D) systems or multi-dimensional scalar equations.

In many realistic scenarios, uncertainties--due to environmental noise, parameter fluctuations, or unmodeled dynamics--cannot be ignored. Such randomness is naturally incorporated through stochastic forcing, leading to systems of stochastic hyperbolic equations  (e.g., \cite{Chapron25,Chow15, Da Prato14, Roach12}). The presence of noise not only affects the well-posedness and stability but also fundamentally alters the controllability properties, as the control must counteract both deterministic and stochastic dynamics. A key question arises:  Can we steer a stochastic multi-dimensional hyperbolic system from any initial state to any desired target state in finite time, and if so, under what conditions? 

In this paper, we address this question for a class of  stochastic first-order symmetric hyperbolic systems in a bounded spatial domain $G \subset \mathbb{R}^n$.  We extend the study of controllability from the deterministic 1-D and scalar multi-dimensional settings to the stochastic multi-dimensional system framework. Notably, a by-product of our analysis is a new controllability result for deterministic first-order multi-dimensional symmetric hyperbolic systems under a suitable structural assumption.

Before describing the control problem studied in this paper, let us first introduce some notations.

Let $ T > 0 $ be a fixed time horizon, and let $(\Omega, \mathcal{F}, \mathbf{F}, \mathbb{P})$ be a complete filtered probability space on which a one‑dimensional Brownian motion $W(\cdot)$ is defined. Here, $\mathbf{F} = \{\mathcal{F}_t\}_{t \in [0,T]}$ denotes the natural filtration generated by $W(\cdot)$ and augmented by all $\mathbb{P}$-null sets. We write $\mathbb{F}$ for the progressive $\sigma$-algebra associated with $\mathbf{F}$.

Let $\mathcal{H}$ be a Banach space. We introduce the following spaces of random variables and stochastic processes:
\begin{itemize}
\item For any $t \in [0,T]$, $L^2_{\mathcal{F}_t}(\Omega;\mathcal{H})$ denotes the Banach space of all $\mathcal{H}$-valued $\mathcal{F}_{t}$-measurable random variables $f$ such that $\mathbb{E} \|f\|_{\mathcal{H}}^2 < \infty$.

\item $L^2_{\mathbb{F}}(0,T;\mathcal{H})$ denotes the Banach space of all $\mathcal{H}$-valued, $\mathbb{F}$-adapted processes $X(\cdot)$ such that
$
\mathbb{E} \|X(\cdot)\|_{L^2(0,T;\mathcal{H})}^2 < \infty.
$

\item $L^\infty_{\mathbb{F}}(0,T;\mathcal{H})$ denotes the Banach space of all essentially bounded $\mathcal{H}$-valued, $\mathbb{F}$-adapted processes.

\item $L^2_{\mathbb{F}}(\Omega;C([0,T];\mathcal{H}))$ denotes the space of all $\mathcal{H}$-valued, $\mathbb{F}$-adapted continuous processes $X(\cdot)$ satisfying
$
\mathbb{E} \|X(\cdot)\|_{C([0,T];\mathcal{H})}^2 < \infty.
$

\item $C_{\mathbb{F}}([0,T];L^2(\Omega;\mathcal{H}))$ denotes the space of $\mathcal{H}$-valued, $\mathbb{F}$-adapted processes whose paths are continuous as maps into $L^2(\Omega;\mathcal{H})$. Analogously, one defines the spaces $C^k_{\mathbb{F}}([0,T];L^2(\Omega;\mathcal{H}))$ for any positive integer $k$ and $C_{\mathbb{F}}([0,T];L^\infty(\Omega;\mathcal{H}))$.
\end{itemize}
All these spaces are equipped with their standard norms. Throughout the paper, for any vector $\mathbf{v}$ we set $|\mathbf{v}| \deq \sqrt{\langle \mathbf{v}, \mathbf{v} \rangle}$, where $\langle \cdot , \cdot \rangle$ denotes the standard Euclidean inner product.

Let $ G \subset \mathbb{R}^n $ ($ n \in \mathbb{N} $) be a bounded domain with a $ C^2 $-smooth boundary $ \Gamma$. Given $ \tau \in (0,T] $, we introduce the following space-time domains:\vspace{-2mm}
$$
Q_\tau \deq (0,\tau) \times G, \qquad
\Sigma_\tau \deq (0,\tau) \times \Gamma.\vspace{-2mm}
$$
For $ i = 1,2,\dots,n $, let $ A_i(x) $ be $ N \times N $ symmetric matrices whose entries belong to $ C^1(\overline{G}) $. Let $\nu(x) = (\nu_1(x),\dots,\nu_n(x))$ denote the unit outward normal vector of $G$ at a boundary point $x \in \Gamma$. By symmetry, there exists an orthogonal matrix $\Pi = \Pi(x)$ such that\vspace{-2mm}
\begin{equation} \label{boundarymatrix}
\Lambda(x) = \Pi^{-1}(x) \Bigl(\sum_{i=1}^n\nu_i(x) A_i(x) \Bigr) \Pi(x)
= \begin{pmatrix}
\Lambda_+(x) & 0 & 0 \\
0 & 0 & 0 \\
0 & 0 & \Lambda_-(x)
\end{pmatrix},\vspace{-2mm}
\end{equation}
where $\Lambda_+(x) \in \mathbb{R}^{n_{+}\times n_{+}}$ and $\Lambda_-(x) \in \mathbb{R}^{n_{-} \times n_{-}}$ are diagonal matrices containing the positive and negative eigenvalues of $\sum_{i=1}^n \nu_i(x) A_i(x)$, respectively. The numbers $n_{+} = n_{+}(x)$ and $n_{-} = n_{-}(x)$ may depend on the boundary point $x \in \Gamma$.

Using $\Pi(x)$ from \eqref{boundarymatrix}, we set\vspace{-2mm}
\begin{equation} \label{zeta}
\zeta(t,x) = \begin{pmatrix}
\zeta_+(t,x) \\
\zeta_0(t,x) \\
\zeta_-(t,x)
\end{pmatrix} \deq \Pi^{-1}(x)\, y(t,x).\vspace{-2mm}
\end{equation}
Consider the following controlled system:\vspace{-2mm}
\begin{equation}\label{fq}
\begin{cases}
dy + \sum_{i=1}^{n} A_i  y_{x_i}\, dt
= (B_1 y + B_3 v)\, dt + (B_2 y + v)\, dW(t) & \text{in } Q_T,\\[4pt]
\zeta_{-} = u & \text{on } \Sigma_T,\\[2pt]
y(0) = y_0 & \text{in } G,
\end{cases}\vspace{-2mm}
\end{equation}
with initial data $y_0 \in L^2(G;\mathbb{R}^N)$ and coefficient matrices $B_1,B_2,B_3 \in L_{\mathbb{F}}^\infty(0,T;L^\infty(G;\mathbb{R}^{N\times N}))$. Here $y(\cdot)$ is the state, and
$(u,v) \in L_{\mathbb{F}}^2\bigl(0,T; L^2(\Gamma; \mathbb{R}^{n_{-}})\bigr) \times
L_{\mathbb{F}}^2\bigl(0,T; L^2(G;\mathbb{R}^N)\bigr)$ are the control inputs.  Here the sample point $ \omega $ will be omitted when no confusion arises. Throughout the paper, the subscript $ x_i $ denotes the partial derivative with respect to $ x_i $.

The functions $\zeta_+(t,x) \in \mathbb{R}^{n_{+}}$ and $\zeta_-(t,x) \in \mathbb{R}^{n_{-}}$ are called the outgoing and  incoming variables of system \eqref{fq} at $x \in \Gamma$, respectively.
The boundary condition of system \eqref{fq} is imposed on the incoming variable $\zeta_-(t,x)$.

Owing to the smoothness of $A_i(x)$ and $\Gamma$, it follows from \cite[pp. 99--101]{K13} that $\Pi(x)$ depends $C^1$-smoothly on $x \in \Gamma$. Furthermore, we assume that $\Gamma $ can be partitioned into finitely many measurable pieces on which $n_{+}(x)$ and $n_{-}(x)$ are constants. Under this assumption, $L_{\mathbb{F}}^2\bigl(0,T; L^2(\Gamma ; \mathbb{R}^{n_{-}})\bigr)$ is a well-defined Hilbert space.

The control system \eqref{fq} is a first‑order stochastic hyperbolic system with a non‑homogeneous boundary condition. Solutions to \eqref{fq} will be understood in the transposition sense. For this purpose, we introduce the corresponding backward stochastic hyperbolic system. For $\tau \in (0,T]$, consider\vspace{-2mm}
\begin{equation} \label{eqAdjoint}
\begin{cases}
\begin{aligned}
&dz\! +\! \sum_{i=1}^n\! A_i z_{x_i}  dt
\!	= \!-\Big[ \Big( B_1^{\top} \!-\! B_2^{\top} B_3^{\top} \!+\! \sum_{i=1}^n\! A_{i,x_i} \Big) z\! +\! B_2^{\top} Z \Big] dt \!
+ \big( Z \!-\! B_3^{\top} z \big) dW(t) && \text{in } Q_\tau, \\[2pt]
&\xi_+ = 0 && \text{on } \Sigma_\tau, \\[2pt]
& z(\tau) = z_\tau && \text{in } G,
\end{aligned}
\end{cases}\vspace{-2mm}
\end{equation}
where $z_\tau \in L_{\mathcal F_\tau}^2(\Omega; L^2(G;\mathbb{R}^{N}))$ and\vspace{-2mm}
$$
\xi(t,x) = \begin{pmatrix}
\xi_+(t,x) \\[1pt]
\xi_0(t,x) \\[1pt]
\xi_-(t,x)
\end{pmatrix} \deq \Pi^{-1}(x) z(t,x),\vspace{-2mm}
$$
with $\xi_+(t,x) \in \mathbb{R}^{n_{+}}$ and $\xi_-(t,x) \in \mathbb{R}^{n_{-}}$. Throughout the paper, the superscript $^\top$ denotes the transpose of a vector or matrix.

By \cite[Theorem 4.10]{LZ21} and the discussion in Remark \ref{RmkA1} of Appendix \ref{appendixA}, we have the following well-posedness result for the backward stochastic system \eqref{eqAdjoint}:
\begin{proposition} \label{wellposednessofbackward}
For any $\tau \in (0,T]$ and any terminal datum $z_\tau \in L_{\mathcal{F}_\tau}^2\bigl(\Omega; L^2(G;\mathbb{R}^{N})\bigr)$, equation \eqref{eqAdjoint} admits a unique mild solution
$(z,Z) \in L_{\mathbb{F}}^2\bigl(\Omega; C([0,\tau]; L^2(G;\mathbb{R}^{N}))\bigr) \times L_{\mathbb{F}}^2\bigl(0,\tau; L^2(G;\mathbb{R}^{N})\bigr)$ satisfying
$$
\|z\|_{L_{\mathbb{F}}^2(\Omega;C([0,\tau];L^2(G;\mathbb{R}^{N})))}
+ \|Z\|_{L_{\mathbb{F}}^2(0,\tau;L^2(G;\mathbb{R}^{N}))}
\leq C \|z_\tau\|_{L_{\mathcal{F}_\tau}^2(\Omega;L^2(G;\mathbb{R}^{N}))}.
$$
\end{proposition}
Hereafter,  unless explicitly stated otherwise, $C$ denotes a generic positive constant that depends at most on $T$, the matrices $A_i$, the domain $G$, and the coefficients $B_1, B_2, B_3$; its value may vary from line to line.

Note that Proposition \ref{wellposednessofbackward} does not directly imply $\xi_{-} \in L^2_{\mathbb{F}}(0,T;L^2(\Gamma ;\mathbb{R}^{n_{-}}))$. This regularity is, however, provided by the following hidden regularity result.
\begin{proposition} \label{hidden}
There exists a constant $C > 0$, independent of $\tau$, such that for any $\tau \in (0,T]$ and any $z_\tau \in L_{\mathcal{F}_\tau}^2\bigl(\Omega; L^2(G;\mathbb{R}^{N})\bigr)$, the solution of the backward stochastic system \eqref{eqAdjoint} satisfies\vspace{-2mm}
$$
\| \xi_{-} \|^2_{L_{\mathbb{F}}^2(0,\tau;L^2(\Gamma ;\mathbb{R}^{n_{-}}))}
\le C \| z_{\tau} \|^2_{L_{\mathcal{F}_{\tau}}^2(\Omega;L^2(G;\mathbb{R}^{N}))}.
$$
\end{proposition}
Using the well-posedness and hidden regularity results for \eqref{eqAdjoint}, we can now define transposition solutions of \eqref{fq} as follows.
\begin{definition} \label{defDefinitionTrans}
A stochastic process $y \in C_{\mathbb{F}}\bigl([0,T]; L^2(\Omega;L^2(G;\mathbb{R}^{N}))\bigr)$ is called a  {\it transposition solution}  of \eqref{fq} if, for every $\tau \in (0,T]$ and every $z_{\tau} \in L_{\mathcal{F}_\tau}^2\bigl(\Omega; L^2(G;\mathbb{R}^{N})\bigr)$,
\begin{equation} \label{deftrans}
\begin{aligned}
\mathbb{E}\int_{G} \langle y(\tau), z_\tau \rangle\,dx
- \int_{G} \langle y_0, z(0) \rangle\,dx
= \mathbb{E}\int_0^\tau \int_{G} \langle v, Z \rangle\,dx\,dt
- \mathbb{E}\int_0^\tau \int_{\Gamma} \langle \xi_{-}, \Lambda_{-} u \rangle\,d\Gamma\,dt,
\end{aligned}
\end{equation}
where $(z,Z)$ is the solution of \eqref{eqAdjoint}.
\end{definition}

In Section \ref{S3} we shall establish the following well-posedness result for system \eqref{fq}.
\begin{proposition} \label{wellposednessfq}
For any initial datum $y_0 \in L^2(G;\mathbb{R}^{N})$, boundary control $u \in L_{\mathbb{F}}^2(0,T;L^2(\Gamma ;\mathbb{R}^{n_{-}}))$ and internal control $v \in L_{\mathbb{F}}^2(0,T;L^2(G;\mathbb{R}^{N}))$, the system \eqref{fq} admits a unique transposition solution
$y \in C_{\mathbb{F}}\bigl([0,T]; L^2(\Omega;L^2(G;\mathbb{R}^{N}))\bigr)$ satisfying\vspace{-2mm}
$$
\| y \|_{C_{\mathbb{F}}([0,T];L^2(\Omega;L^2(G;\mathbb{R}^{N})))}
\leq C\Big( \| y_0\|_{L^2(G;\mathbb{R}^{N})}
+ \| u\|_{L_{\mathbb{F}}^2(0,T;L^2(\Gamma ;\mathbb{R}^{n_{-}}))}
+ \| v\|_{L_{\mathbb{F}}^2(0,T;L^2(G;\mathbb{R}^{N}))}\Big).
$$
\end{proposition}
We now recall the notion of exact controllability for system \eqref{fq}.
\begin{definition}
System \eqref{fq} is called  {\it exactly controllable} at time $T$ if, for every initial state $y_0 \in L^2(G;\mathbb{R}^{N})$ and every target state $y_1 \in L^{2}_{\mathcal{F}_{T}}\bigl(\Omega; L^{2}(G;\mathbb{R}^{N})\bigr)$, there exist boundary control
$u \in L^{2}_{\mathbb{F}}\bigl(0,T; L^{2}(\Gamma ; \mathbb{R}^{n_{-}})\bigr)$ and an internal control $v \in L^{2}_{\mathbb{F}}\bigl(0,T; L^{2}(G; \mathbb{R}^{N})\bigr)$
such that the corresponding solution $y(\cdot)$ of \eqref{fq} satisfies $y(T) = y_1$, $\mathbb{P}$-a.s.
\end{definition}

The main difficulty in studying exact controllability problem of System \eqref{fq} lies in obtaining a suitable observability estimate  for the associated backward stochastic adjoint system, which through duality is equivalent to exact controllability. For deterministic hyperbolic systems, such estimates are often derived via multiplier methods or geometric optics. In the stochastic setting, however, these techniques do not directly apply due to the It\^o correction and the non-anticipative nature of solutions. The primary technical novelty of this work is the development of a  global Carleman estimate  tailored for the adjoint system \eqref{eqAdjoint}. Carleman estimates have proven powerful for control problems of stochastic PDEs (see, e.g., \cite{Fu17-1, Fu17-2, LiaoLu24, Lx19, Lu14, Lu23, LuW22, LZ21, Z10, Zhao25} and the references therein).   In particular, \cite{Lu14} employed a Carleman estimate to prove exact controllability for a stochastic transport equation. However, Carleman estimates for systems in more than two variables remain challenging to derive and are still relatively scarce in the literature. Our global Carleman estimate is inspired by the one established for deterministic first‑order symmetric hyperbolic systems in \cite{FTY22}.

To establish the exact controllability of system \eqref{fq}, we impose the following structural assumption on the coefficient matrices $A_i(x)$ ($i = 1,2,\dots,n$).
\begin{condition}
\label{cond1}
There exists a function $\eta(\cdot) \in C^{1}(\overline{G})$ and a constant $c_{0} > 0$ such that\vspace{-2mm}
\begin{equation}\label{con1}
- \Big\langle \Big(\sum_{i=1}^{n} A_{i}(x) \eta_{x_{i}}(x) \Big) \gamma , \gamma \Big \rangle
\ge c_{0} |\gamma|^2,
\qquad
\forall\, x \in \overline{G},\; \gamma \in \mathbb{R}^N .\vspace{-2mm}
\end{equation}
\end{condition}
Condition \ref{cond1} is essential for the derivation of our Carleman estimate. Its geometric interpretation as well as several physical examples that satisfy this condition will be discussed in  Section \ref{S2}. We are now ready to state our main exact controllability result.
\begin{theorem}
\label{thmControllability}
Assume that Condition \ref{cond1} holds. If the terminal time $T$ satisfies\vspace{-2mm}
\begin{align}\label{eqMinimumTime}
T > T_{0} \deq \frac{1}{c_{0}}\Big(\max_{x \in \overline{G}} \eta(x) - \min_{x \in \overline{G}} \eta(x)\Big),\vspace{-2mm}
\end{align}
then system \eqref{fq} is exactly controllable at time $T$.
\end{theorem}
\begin{remark}
Consider the scalar case $N = 1$. Let $\mathcal{O} = (A_{1}, \dots, A_{n})^{\top} \in \mathbb{R}^{n}$, $\Pi = 1$, and define the inflow boundary by
$\Gamma^{-} = \{ x \in \Gamma \mid \mathcal{O} \cdot \nu(x) < 0 \}$.
Take $\tilde{a}_1, \tilde{a}_2 \in \overline{G}$ such that $|\tilde{a}_1 - \tilde{a}_2| = \max_{a_1,a_2 \in \overline{G}} |a_1 - a_2|$.
Without loss of generality, we may assume $|\mathcal{O}| = 1$, $0 \in G$, and $0 = \tilde{a}_1 + \tilde{a}_2$.

Then \eqref{fq} reduces to the scalar stochastic transport equation\vspace{-2mm}
$$
\begin{cases}
d y + \mathcal{O} \cdot \nabla y \, dt = (B_{1} y + B_{3} v) \, dt + (B_{2} y + v) \, dW(t), & \text{in } Q_T, \\[2pt]
y = u, & \text{on } (0,T) \times \Gamma^{-}, \\[2pt]
y(0) = y_{0}, & \text{in } G.
\end{cases}\vspace{-2mm}
$$
This system was studied in \cite[Chapter 8]{LZ21}, where exact controllability was proved under the condition $T > \widetilde{T}_{0} \deq 2 \max_{x \in \Gamma} |x|$.

Choosing the weight function $\eta(x) = - \mathcal{O} \cdot x$ in Condition \ref{cond1} yields $c_{0} = 1$. Consequently, Theorem \ref{thmControllability} guarantees exact controllability whenever\vspace{-2mm}
$$
T > T_{0} = \max_{x \in \overline{G}} (\mathcal{O} \cdot x) - \min_{x \in \overline{G}} (\mathcal{O} \cdot x).\vspace{-2mm}
$$
Observe that $T_{0} \le \widetilde{T}_{0}$ in general; hence our result improves the previously known sufficient time. Moreover, $T_{0}$ is indeed the  minimal  control time in this setting. If controllability were possible for some $T < T_{0}$, one could deduce exact controllability for a deterministic transport equation in a time shorter than the characteristic crossing time, which is impossible because certain characteristics would not reach the boundary within the horizon, violating the exact controllability property.
\end{remark}

In the study of exact controllability for scalar stochastic transport equations in \cite{Lu14}, the author presented several results on the  lack  of exact controllability, which demonstrate the necessity of employing both boundary and internal controls. Similar controllability limitations hold for the more general stochastic first‑order hyperbolic system \eqref{fq}.

We use two controls to achieve exact controllability  of the control system  \eqref{fq}. To underscore the necessity of the proposed control structure, we also establish several negative controllability results:  
\begin{theorem}\label{thm1.2}
Assume additionally that $B_1\in L^{\infty}(0,T;L^{\infty}(G;\mathbb{R}^N))$ and $B_3 = 0$. If $u \equiv 0$, then system \eqref{fq} is  not  exactly controllable for any $T > 0$.
\end{theorem}
\begin{theorem}\label{thm1.3}
Assume additionally that $B_2 \in C_{\mathbb{F}}([0,T]; L^{\infty}(\Omega; L^{\infty}(G; \mathbb{R}^{N \times N})))$.
Let $G_0$ be an open subset of $G$ with $G \setminus \overline{G_0} \neq \emptyset$. If the control $v$ is supported only in $(0,T) \times G_0$, then system \eqref{fq} is  not  exactly controllable for any $T > 0$.
\end{theorem}

Furthermore, control must also appear in the diffusion term. To see this, consider the following modified system:\vspace{-2mm}
\begin{equation}\label{1.12}
\begin{cases}
\begin{aligned}
& dy + \sum_{i=1}^{n} A_i(x) y_{x_i} \, dt
= \big( B_1 y + \tilde{v} \big) dt + B_2 y \, dW(t) && \text{in } Q_T,\\[4pt]
& \zeta_{-} = u && \text{on } \Sigma_T,\\[2pt]
& y(0) = y_0 && \text{in } G,
\end{aligned}
\end{cases}\vspace{-2mm}
\end{equation}
where $B_1, B_2, u$ are as in \eqref{fq} and $\tilde{v} \in L_{\mathbb{F}}^2(0,T; L^2(G; \mathbb{R}^N))$ is an internal control acting only on the drift. For this system we obtain the following negative result.
\begin{theorem}\label{thm1.4}
Assume additionally that $B_2 \in C_{\mathbb{F}}([0,T]; L^{\infty}(\Omega; L^{\infty}(G; \mathbb{R}^{N \times N})))$.
Then system \eqref{1.12} is  not  exactly controllable for any $T > 0$.
\end{theorem}
\begin{remark}
We emphasize that, because we are dealing with a system (coupled components), it may be possible under additional structural assumptions to achieve controllability with fewer controls; see e.g., \cite{HO21, ABCO17}. Determining the minimal number of required controls lies beyond the scope of the present work.
\end{remark}

\medskip

The rest of the paper is organized as follows. In Section \ref{S2}, we give a geometric interpretation of Condition \ref{cond1} and list several relevant examples. In Section \ref{S3}, we prove the hidden regularity (Proposition \ref{hidden}) for the solution to \eqref{eqAdjoint} and establish the well-posedness (Proposition \ref{wellposednessfq}) for the system \eqref{fq}. Section \ref{S4} is devoted to proving Theorem \ref{thmControllability}. In Section \ref{S5}, we prove Theorems \ref{thm1.2}--\ref{thm1.4}. Finally, in Appendix \ref{appendixA}, we give the proofs of some technical results.

\section{The geometric interpretation of Condition \ref{cond1}}\label{S2}

In this section, we provide a geometric interpretation of Condition \ref{cond1} and present several examples satisfying this condition.

Put\vspace{-2mm}
$$
\mathscr{L}\deq \partial_{t}+\sum_{i=1}^{n}A_{i}(x)\partial_{x_{i}},\vspace{-2mm}
$$
the principal symbol of operator $\mathscr{L}$ is\vspace{-2mm}
$$
P(t,x, \tau, \varpi) = \tau I_N + \sum_{i=1}^n \varpi_i A_i(x),\vspace{-2mm}
$$
where
$\tau \in \mathbb R$, $\varpi=(\varpi_1,\cdots,\varpi_n) \in \mathbb R^n$ are the variables on frequency space, and $ I_{N}$ is the identity matrix in $ \mathbb{R}^{N \times N} $. The propagation of information is on a sub-manifold of the phase space (cotangent bundle) $\mathcal T^*(\mathbb{R} \times G)$:\vspace{-2mm}
\begin{equation}
\text{Char}(P)  \deq  \Big\{ (t, x, \tau, \varpi) \in \mathcal T^*(\mathbb{R}
\times G) \setminus \{0\} \;\Big|\; \det\Big(\tau I_N + \sum_{i=1}^n \varpi_i A_i(x)\Big) = 0 \Big\}.\vspace{-2mm}
\end{equation}
Recalling that $A_i(x)$ $(i=1,\cdots,n)$ are symmetric, hence we have\vspace{-2mm}
$$
\det\Big(\tau I_N + \sum_{i=1}^n \varpi_i A_i(x)\Big) = \prod_{k=1}^N (\tau + \lambda_k(x, \varpi))\vspace{-2mm}
$$
with $\lambda_k=\lambda_k(x,\varpi) \in \mathbb R$ being the $k$-th eigenvalue of the symmetric matrix $\sum_{i=1}^n \varpi_i A_i(x).$ For each $k$, we define the $k$-th Hamiltonian\vspace{-2mm}
$$
H_k=H_k(t,x,\tau,\varpi) \deq \tau+\lambda_{k}(x,\varpi).\vspace{-2mm}
$$
The corresponding $k$-th bicharacteristics
$(t_k(s),\Bx_k(s),\tau_k(s), \boldsymbol{\varpi}_k(s))\subset \text{Char}(P)$ (where the bold letters are used to distinguish them from the space and frequency variables)
are the integral curves of the $k$-th Hamiltonian vector field:\vspace{-2mm}
$$
\begin{cases}
\begin{aligned}
&\frac{dt_k}{ds}=\frac{\partial H_k}{\partial \tau}=1,\\
&\frac{d\Bx_k}{ds}=\nabla_{\varpi}H_k=\nabla_\varpi \lambda_k(\Bx_{k},\boldsymbol{\varpi}_{k}),\\
&\frac{d\tau_k}{ds}=-\frac{\partial H_k}{\partial t}=0,\\
&\frac{d\boldsymbol{\varpi}_k}{ds}=-\nabla_xH_k=-\nabla_x \lambda_k(\Bx_k,\boldsymbol{\varpi}_{k}).
\end{aligned}
\end{cases}\vspace{-2mm}
$$
The curves $(t_k(s),\Bx_k(s))$
are called rays which represent the trajectory of energy propagation (wave front) in the physical space. The following proposition clarifies the geometrical meaning of Condition \ref{cond1}: the existence of a potential function $\eta(\cdot)$ such that all rays propagate along the direction of strict descent of $\eta$.

\begin{proposition} \label{etadecay}
Under Condition \ref{cond1},  $\eta(x)$ strictly decreases along
$\Bx_k(s)$ for all $k=1,\cdots,N$, with a uniform decay rate at least $c_0$ (where $c_0$ is the constant in Condition \ref{cond1}):  \vspace{-2mm}
\begin{equation} \label{geometric}
\frac{d}{ds} \eta(\Bx_k(s)) \leq -c_0, \quad \forall k\in\{1,\cdots,N\}.\vspace{-2mm}
\end{equation}
\end{proposition}
\begin{proof}
Noting that $\sum_{i=1}^n\varpi_iA_i(x)$ is symmetric, then for any $k\in\{1,\cdots,N\}$, there exists an  eigenvector $r_k=r_k(x,\varpi)$
with $|r_k|=1$ so that\vspace{-2mm}
$$
\Big(\sum_{i=1}^n \varpi_iA_i(x)\Big)r_k=\lambda_k(x,\varpi)r_k.\vspace{-2mm}
$$
Taking derivative with respect to $\varpi_j$ on both sides of the above identity, and then taking the inner product with $r_k$, we obtain that\vspace{-2mm}
$$
\begin{aligned}
\langle r_k, A_j(x)r_k \rangle +\Big\langle r_k,\Big(\sum_{i=1}^n \varpi_iA_i(x)\Big)\frac{\partial r_k}{\partial \varpi_j} \Big\rangle =\Big\langle r_k, \frac{\partial \lambda_k}{\partial \varpi_j}r_k \Big\rangle+\lambda_k(x,\varpi) \Big\langle r_k,\frac{\partial r_k}{\partial \varpi_j}\Big\rangle.
\end{aligned}\vspace{-2mm}
$$
By the symmetry and  the fact that $|r_k|=1$, we have\vspace{-2mm}
$$
\langle r_k, A_j(x)r_k \rangle=\frac{\partial \lambda_k}{\partial \varpi_j}.\vspace{-2mm}
$$
With the preparations above, we compute the evolution of $\eta$ along the $k$-th rays:\vspace{-2mm}
\begin{equation*}
\frac{d}{ds} \eta(\Bx_k(s)) = \Big\langle \nabla \eta ,\frac{d \Bx_k}{ds} \Big\rangle =\langle \nabla \eta, \nabla_\varpi \lambda_k \rangle
= \sum_{i=1}^n \eta_{x_i}  \langle r_k, A_i(x)r_k \rangle
= \Big\langle r_k, \Big( \sum_{i=1}^n A_i(x) \eta_{x_i} \Big) r_k \Big\rangle.\vspace{-2mm}
\end{equation*}
This, together with \eqref{con1} and $|r_k| =1$, implies \eqref{geometric} immediately.
\end{proof}

\medskip

Proposition \ref{etadecay} implies that all rays are forced to escape through the boundary $\Gamma $ within a finite time, since the decrease rate of $\eta(x)$ along the rays is bounded from below by $c_0$. We can estimate the maximum time required for the slowest rays to traverse the domain:\vspace{-2mm}
$$T_0 = \frac{1}{c_0}\left(\max_{x \in \overline{G}} \eta(x) - \min_{x \in \overline{G}} \eta(x)\right).\vspace{-2mm}
$$
When $T > T_0$, all rays have sufficient time to propagate from the interior to the boundary, allowing them to be controlled by the boundary inputs $u$.

\begin{remark}
Condition \ref{cond1} is somewhat restrictive. In particular, several classical models may fail to satisfy it when rewritten as first-order hyperbolic systems, such as the wave equation and other wave-like equations (e.g., Maxwell or Klein--Gordon). In these settings, one typically needs to design an appropriate multiplier case-by-case; see \cite{LZ21,Z00}.
\end{remark}

At the end of this section, we illustrate the applicability of Condition \ref{cond1} with several representative models.

\begin{example}[Stochastic free traffic flow]
In traffic flow modeling, the Aw-Rascle-Zhang system constitutes a first-order hyperbolic model defined on a one‑dimensional road, with macroscopic state variables given by the traffic density $\rho$ and the velocity $v$. Linearizing around a constant equilibrium $(\rho_0,v_0)$ (see \cite[Section 1.10]{BC16}) yields the deterministic system\vspace{-2mm}
$$
y_t+A_1y_x = B_1y ,\qquad y=(\rho ,v)^\top,\vspace{-2mm}
$$
where\vspace{-2mm}
$$
A_1=\begin{pmatrix}
	v_0 & \rho_0\\[2pt]
	0 & v_0-\rho_0P'(\rho_0)
\end{pmatrix},\qquad 
B_1=\sigma\begin{pmatrix}
	0 & 0\\[2pt]
	V'(\rho_0) & -1
\end{pmatrix}.\vspace{-2mm}
$$
Here $P(\rho)$ (respectively $V(\rho)$) is a strictly increasing (resp. decreasing) function representing the traffic pressure (resp. average speed), and $\sigma$ is proportional to the drivers' reaction time. The eigenvalues of $A_1$ are $v_0$ and $v_0-\rho_0P'(\rho_0)$. Assuming $v_0>\rho_0P'(\rho_0)>0$ corresponds to a free‑flow regime, i.e. the road is relatively empty; consequently both eigenvalues are positive.

In reality, driver behaviour is uncertain, causing $\sigma$ to fluctuate randomly, while additional internal stochastic perturbations may also affect the dynamics \cite[Chapter 5]{YK22}. Therefore we consider the stochastic version\vspace{-2mm}
$$
dy+A_1y_x\,dt = B_1(t,x)y\,dt + B_2(t,x)y\,dW(t).\vspace{-2mm}
$$
Although $A_1$ itself is not symmetric, we introduce the symmetric positive‑definite matrix\vspace{-2mm}
$$
A_0=\begin{pmatrix}
	P'(\rho_0) & 1\\[2pt]
	1 & 2/P'(\rho_0)
\end{pmatrix},\vspace{-2mm}
$$
for which $A_0A_1$ becomes symmetric. Setting $\tilde y = A_0^{1/2}y$ gives\vspace{-2mm}
$$
d\tilde y+\tilde A_1\tilde y_x\,dt = \tilde B_1(t,x)\tilde y\,dt+\tilde B_2(t,x)\tilde y\,dW(t),\vspace{-2mm}
$$
where $\tilde A_1 = A_0^{1/2}A_1A_0^{-1/2}$ is symmetric. Because the eigenvalues of $\tilde A_1$ are positive, Condition  \ref{cond1} is satisfied with the choice $\eta(x)=-x$.
\end{example}

\begin{example}[Linear stochastic age‑structured SIR epidemiological equations]
In mathematical epidemiology, populations subject to an infectious disease are commonly stratified into three compartments:
\begin{itemize}
	\item \textbf{Susceptible} ($S$): individuals who are not infected but may contract the disease;
	\item \textbf{Infected} ($I$): individuals currently carrying the pathogen;
	\item \textbf{Recovered} ($R$): individuals who have recovered and acquired permanent immunity.
\end{itemize}
Let $S(t,x),\,I(t,x),\,R(t,x)$ denote the age‑density of susceptible, infected, and recovered individuals at time $t$ and age $x$. Under standard modelling assumptions (see, e.g., \cite[Section 1.12]{BC16}), the spread of the disease in a closed population can be described by a deterministic linear hyperbolic system with coefficient matrix $A_1(x)=I_3$.

Deterministic models, however, presume fixed parameters. In practice, several sources of randomness cannot be ignored (cf. \cite{Allen2008}):
\begin{enumerate}
	\item Contact rates vary seasonally and with behavioural or policy changes;
	\item Detection rates and compliance with isolation measures are uncertain;
	\item Viral mutations introduce randomness in transmissibility;
	\item Environmental factors (temperature, humidity) affect transmission stochastically.
\end{enumerate}
To incorporate such inherent uncertainty, stochastic age‑structured SIR equations include noise terms driven by Brownian motion:\vspace{-2mm}
\begin{equation}\label{SSIReq}
	\begin{cases}
		\displaystyle
		dS + S_x dt = -\Bigl(\beta(x,t)\frac{SI}{N}+\mu_S(x)S\Bigr)dt + \sigma_S(t)\,dW(t),\\[6pt]
		\displaystyle
		dI + I_x dt = \Bigl(\beta(x,t)\frac{SI}{N}-\bigl(\gamma(x)+\mu_I(x)\bigr)I\Bigr)dt + \sigma_I(t)\,dW(t),\\[6pt]
		\displaystyle
		dR + R_x dt = \bigl(\gamma(x)I-\mu_R(x)R\bigr)dt + \sigma_R(t)\,dW(t).
	\end{cases}\vspace{-2mm}
\end{equation}
Here $\partial/\partial x$ represents ageing; each $\sigma_j(t)dW(t)$ is a Brownian‑driven stochastic perturbation; $\beta(x,t)$ denotes the (possibly time‑varying) age‑specific transmission rate; $\gamma(x)$ is the age‑dependent recovery rate; and $\mu_\cd(x)$ are age‑dependent mortality rates.

For system \eqref{SSIReq}, the choice $\eta(x)=-x$ satisfies Condition \ref{cond1}. 
\end{example}

\begin{example}[Stochastic shallow‑water equations for torrential flow] \label{shallowwater}
The shallow‑water equations are widely used in hydraulic engineering for the regulation of navigable rivers and irrigation networks \cite{HS19}, and in geophysics for the simulation of atmospheric and oceanic flows \cite{HT14}.

Consider the 2‑D linearized shallow‑water equations \cite{YY24} around a constant equilibrium $(H_0,U_0,V_0)$. The corresponding coefficient matrices are\vspace{-2mm}
\begin{equation}\label{2.3}
A_1=
\begin{pmatrix}
U_0 & \sqrt{g H_0} & 0 \\[2pt]
\sqrt{g H_0} & U_0 & 0 \\[2pt]
0 & 0 & U_0
\end{pmatrix},
\qquad
A_2 =
\begin{pmatrix}
V_0 & 0 & \sqrt{g H_0}\\[2pt]
0 & V_0 & 0 \\[2pt]
\sqrt{g H_0} & 0 & V_0
\end{pmatrix},\vspace{-2mm}
\end{equation}
where $g$ is the gravitational constant, and the constants $H_0$, $U_0$, $V_0$ represent the water height and the two horizontal velocity components of the steady flow, respectively.

Assume that
$$
U_0^2+V_0^2 > g H_0,
$$
which means the flow velocity exceeds the gravitational wave speed. This regime is called torrential (or supercritical) flow. In this case, waves cannot propagate upstream against the current; all disturbances and energy are swept strictly downstream.

If a random perturbation, for instance due to stochastic wind stress (see \cite[Chapter 1]{Chapron25}), is present, we obtain the following stochastic shallow‑water equation:
$$
dy + \bigl(A_1 y_{x_{1}} + A_2 y_{x_{2}}\bigr)dt
= B_1(t,x) y \, dt + B_2(t,x) y \, dW(t),
$$
where $y = (h,u,v)^{\top}$ is the symmetrized state variable, and the matrices $A_1,A_2$ are given by \eqref{2.3}.

For constants $\alpha_1,\alpha_2$ satisfying $\alpha_1^2+\alpha_2^2 = 1$, the eigenvalues of the symmetric matrix\vspace{-2mm}
$$
\alpha_1A_1+\alpha_2A_2=
\begin{pmatrix}
\alpha_1U_0+\alpha_2V_0 & \alpha_1\sqrt{g H_0} & \alpha_2\sqrt{g H_0} \\[2pt]
\alpha_1\sqrt{g H_0} & \alpha_1U_0+\alpha_2V_0 & 0 \\[2pt]
\alpha_2\sqrt{g H_0} & 0 & \alpha_1U_0+\alpha_2V_0
\end{pmatrix}\vspace{-2mm}
$$
are\vspace{-2mm}
$$
\begin{aligned}
\lambda_1  = \alpha_1U_0 + \alpha_2V_0,\q
\lambda_2  = (\alpha_1U_0 + \alpha_2V_0) - \sqrt{g H_0},\q
\lambda_3  = (\alpha_1U_0 + \alpha_2V_0) + \sqrt{g H_0}.
\end{aligned}\vspace{-2mm}
$$
Choosing\vspace{-2mm}
$$
(\alpha_1,\alpha_2)=\bigg(\frac{U_0}{\sqrt{U_0^2+V_0^2}},\;
\frac{V_0}{\sqrt{U_0^2+V_0^2}}\bigg)\vspace{-2mm}
$$
makes the matrix $\alpha_1A_1+\alpha_2A_2$ positive definite. Consequently, we may take the weight function\vspace{-2mm}
$$
\eta(x) = -\,\bigl(\alpha_1 x_1 + \alpha_2 x_2\bigr).\vspace{-2mm}
$$
\end{example}

\begin{example}[Stochastic supersonic gas equations]
The evolution of density, velocities and entropy of a compressible, inviscid and non-heat-conducting gas is governed by the Euler equations, see
\cite[(13.2.26), p.~395]{BGS06}.
For $n \ (n=2,3)$ dimensional ideal gas, we linearize the system around a constant equilibrium $(\rho_0,u_0,s_0)$, where the constants $\rho_0$ and $s_0$ are density and entropy, while $u_0=(u_{01},u_{02},\cdots,u_{0n})$ is the $n$-dimensional constant vector standing for the velocity. After symmetrization (see \cite[p.~395]{BGS06}), we denote the symmetric coefficient matrices by $A_i \ (i=1,\cdots,n)$, then\vspace{-2mm}
\begin{equation}\label{2.4}
A_{i}=
\begin{pmatrix}
\frac{u_{0i}}{\rho_0 {\bf c}_0^2} & e_{i}^\top & 0 \\
e_{i} & \rho_0u_{0i} I_n & 0 \\
0 & 0 & u_{0i}
\end{pmatrix},
\quad i=1,\cdots,n,\vspace{-2mm}
\end{equation}
where $e_{i}$ is the $n$-dimensional unit vector with the $i$-th component being 1, and $\Bc_0$ is the sound speed. The uncertain friction parameters naturally render the system stochastic (see \cite{GH23} and references therein):\vspace{-2mm}
$$
dy+\sum_{i=1}^{n}A_{i}y_{x_{i}}dt=B_{1}(t,x)y\,dt+B_{2}(t,x)y\,dW(t),\vspace{-2mm}
$$
where the coefficients $A_i \ (i=1,\cdots,n)$ are given in \eqref{2.4}.

Consider a constant $n$-dimensional unit vector $\alpha=(\alpha_1,\cdots,\alpha_n)^\top$, then the eigenvalues of the $(n+2)\times (n+2)$-matrix\vspace{-2mm}
$$
\sum_{i=1}^n \alpha_iA_i=\begin{pmatrix}
\frac{u_0 \cdot \alpha}{\rho_0 {\bf c}_0^2} & \alpha^\top & 0 \\
\alpha & \rho_0(u_0 \cdot \alpha) I_n & 0 \\
0 & 0 & u_0 \cdot \alpha
\end{pmatrix}\vspace{-2mm}
$$
are\vspace{-2mm}
$$
\lambda_1,\lambda_{n+2}=
\frac{1}{2}\Big(\rho_0(u_0\cdot \alpha)+\frac{u_0\cdot \alpha}{\rho_0 {\bf c}_0^2}\pm \sqrt{\Big(\rho_0(u_0\cdot \alpha)-\frac{u_0\cdot \alpha}{\rho_0 {\bf c}_0^2}\Big)^2+4} \ \Big),\vspace{-2mm}
$$
and\vspace{-2mm}
$$
\lambda_2 = u_0\cdot \alpha,
\qquad
\lambda_3 = \cdots = \lambda_{n+1} = \rho_0(u_0\cdot \alpha).\vspace{-2mm}
$$
Assume the gas speed exceeds the sound speed (supersonic), i.e.,
$|u_0|>{\bf c}_0$,
then we can choose $\alpha=u_0/|u_0|$ and $\eta(x)=-\sum_{i=1}^n \alpha_ix_i$.
\end{example}

\begin{remark}
With minor modifications of the method employed in \cite{Chow15,HT14} and noting \cite{A15}, the well‑posedness result established in Appendix \ref{appendixA} for the systems listed above can be extended to polytopal domains.
Notice that on each point (1D), edge (2D) or face (3D) of such a domain the outward unit normal vector $\nu$ is constant. As a result, the matrix $\sum_{i=1}^{n}\nu_{i}A_{i}$ is constant on each edge/face, and therefore the numbers of positive and negative eigenvalues, denoted respectively by $n_{+}$ and $n_{-}$, are piecewise constant along the boundary $\Gamma$.
\end{remark}

\section{The well-posedness of the control system \eqref{fq}} \label{S3}

In this section, we establish the well-posedness of the control system \eqref{fq}.

Consider the unbounded operator\vspace{-2mm}
$$
\mathcal A \deq -\sum_{i=1}^n A_i(x)\,\partial_{x_i}\vspace{-2mm}
$$
with domain\vspace{-2mm}
$$
D(\mathcal A)=\Big\{\, \mathbf y \in L^2(G;\mathbb R^N) \;\Big|\;
\mathcal{A} \mathbf y \in L^2(G;\mathbb R^N),\; \widetilde{\mathbf y}_{-}=0 \,\Big\},\vspace{-2mm}
$$
where $\widetilde{\mathbf y}_{-}$ is defined through the decomposition\vspace{-2mm}
$$
\widetilde{\mathbf y}= \begin{pmatrix}
\widetilde{\mathbf y}_+,
\widetilde{\mathbf y}_0,
\widetilde{\mathbf y}_-
\end{pmatrix}^\top
= \Pi^{-1}\mathbf y,\vspace{-2mm}
$$
and $\Pi$ is the orthogonal matrix introduced in \eqref{boundarymatrix}.
Here $\widetilde{\mathbf y}_{-}$ and $\widetilde{\mathbf y}_{+}$ are called, respectively, the  incoming  and  outgoing  variables of the operator $\mathcal A$ on the boundary $\Gamma$.

\begin{lemma} \label{C0semigroup}
The operator $\mathcal A$ generates a $C_{0}$-semigroup on $L^{2}(G;\mathbb{R}^{N})$.
\end{lemma}

Lemma \ref{C0semigroup} is essentially presented in \cite[Chapter 9]{BGS06} and \cite[Chapter 7]{CP11}, but the boundary conditions are treated in a different manner from ours. For completeness, we provide a self-contained proof in Appendix \ref{appendixA}.

We first prove Proposition \ref{hidden}.
\begin{proof}[Proof of Proposition \ref{hidden}]
Without loss of generality, we take $\tau = T$. The backward stochastic system \eqref{eqAdjoint} can be written as\vspace{-2mm}
$$
\begin{cases}
\ds dz + \mathcal A^* z \, dt = f \, dt + \big( Z - B_3^\top z \big) dW(t)  &  \text{in } Q_T,\\[4pt]
\ds z(T) = z_T  &  \text{in } G,
\end{cases}\vspace{-2mm}
$$
where $ f = -\big( B_1^\top z - B_2^\top B_3^\top z + B_2^\top Z \big) $ and $\mathcal A^*$ is defined in Lemma \ref{lemmaAdjoint}.

Let $\mu \in \rho(\mathcal A^*) \cap \mathbb R$, where $\rho(\mathcal A^*)$ denotes the resolvent set of $\mathcal A^*$. Set $R^*(\mu) \deq \mu (\mu I - \mathcal A^*)^{-1}$ and consider the regularized backward stochastic system\vspace{-2mm}
\begin{equation} \label{yoshidaapproximation}
\begin{cases}
\ds dz_\mu + \mathcal A^* z_\mu \, dt = R^*(\mu) f \, dt + R^*(\mu) \big( Z_\mu - B_3^\top z_\mu \big) dW(t) & \text{in } Q_T,\\[4pt]
\ds z_\mu(T) = R^*(\mu) z_T & \text{in } G.
\end{cases}\vspace{-2mm}
\end{equation}
Since $R^*(\mu) z_T \in L^2_{\mathcal{F}_T}\bigl(\Omega; D(\mathcal A^*)\bigr)$, by \cite[Theorem 4.12]{LZ21} the system \eqref{yoshidaapproximation} admits a unique solution\vspace{-2mm}
$$
(z_\mu, Z_\mu) \in L_{\mathbb{F}}^2\bigl(\Omega; C([0,T]; D(\mathcal A^*))\bigr) \times L_{\mathbb{F}}^2\bigl(0,T; L^2(G;\mathbb{R}^N)\bigr).\vspace{-2mm}
$$
Let   $\xi_\mu=((\xi_{\mu})_{+},(\xi_{\mu})_{0},(\xi_{\mu})_{-})^{\top}\deq \Pi^{-1} z_\mu$.
Applying It\^o's formula and integration by parts gives\vspace{-2mm}
\begin{align}\label{1.23-eq1}
&\mathbb{E} \| R^*(\mu) z_T \|_{L^2(G;\mathbb{R}^N)}^2 - \| z_\mu(0) \|_{L^2(G;\mathbb{R}^N)}^2 \notag\\
&= -\mathbb{E} \int_0^T \int_{\Gamma} \big\langle (\xi_\mu)_-,\; \Lambda_- (\xi_\mu)_- \big\rangle \, d\Gamma \, dt \\
&\quad + \mathbb{E} \int_0^T \int_G \Big[ -\Big\langle z_\mu,\; \big( \sum_{i=1}^n A_{i,x_i} \big) z_\mu \Big\rangle
+ 2 \langle z_\mu, R^*(\mu) f \rangle
+ \big| R^*(\mu) (Z_\mu - B_3^\top z_\mu) \big|^2 \Big] dx \, dt .\notag\vspace{-2mm}
\end{align}
Passing to the limit as $\mu \to \infty$ and using \cite[Theorem 4.12]{LZ21} to ge the limit of all terms in \eqref{1.23-eq1}, we obtain\vspace{-2mm}
$$
\begin{aligned}
&-\mathbb{E} \int_0^T \int_{\Gamma} \langle \xi_-,\; \Lambda_- \xi_- \rangle \, d\Gamma \, dt \\
&\le \mathbb{E} \| z_T \|_{L^2(G;\mathbb{R}^N)}^2
-\mathbb{E} \int_0^T \int_G \Big[ -\Big\langle z,\; \big( \sum_{i=1}^n A_{i,x_i} \big) z \Big\rangle
+ 2 \langle z, f \rangle + \big| Z - B_3^\top z \big|^2 \Big] dx \, dt \\
&\le C \, \mathbb{E} \| z_T \|_{L^2(G;\mathbb{R}^N)}^2,
\end{aligned}\vspace{-2mm}
$$
where the last inequality follows from Proposition \ref{wellposednessofbackward}. This completes the proof.
\end{proof}

Now we are in a position to prove the well-posedness of the control system \eqref{fq}.

\begin{proof}[Proof of Proposition \ref{wellposednessfq}]
We divide the proof into three steps.

\textbf{Step 1. Extension of the boundary control.}

Let $u \in L_{\mathbb F}^2(0,T;L^2(\Gamma ; \mathbb{R}^{n_{-}}))$ be given. We embed this $\mathbb{R}^{n_{-}}$-valued function into $\mathbb{R}^N$ by following the same block structure as $\zeta$ in \eqref{zeta} and setting the extra components to zero. Multiplying the extended vector by the orthogonal matrix $\Pi(x)$ from \eqref{boundarymatrix} yields\vspace{-2mm}
$$
U(t,x)\deq \Pi(x)\begin{pmatrix}0 \\ 0 \\ u(t,x)\end{pmatrix}
\in L_{\mathbb F}^2(0,T;L^2(\Gamma ; \mathbb{R}^N)).\vspace{-2mm}
$$
By a standard approximation argument (temporal truncation and mollification; see e.g., \cite[Section 4.10]{CP11}), there exists a sequence $\{U_m\}_{m=1}^\infty\subset C_{\mathbb F}^1([0,T];L^2(\Omega;H^{1/2}(\Gamma ;\mathbb R^N)))$ with $U_m(0,x)=0$ for all $m$ such that\vspace{-2mm}
$$
\lim_{m\to\infty} U_m = U \quad \text{in } L_{\mathbb F}^2(0,T;L^2(\Gamma ;\mathbb R^N)).\vspace{-2mm}
$$
Let $((\Pi^{-1}U_m)_{+},(\Pi^{-1}U_m)_{0},(\Pi^{-1}U_m)_{-})^{\top}\deq\Pi^{-1}U_m$ be split as in \eqref{zeta}, and define $u_m\deq(\Pi^{-1}U_m)_{-}$.  Then\vspace{-2mm}
\begin{equation} \label{convergenceofu}
\lim_{m\to\infty} u_m = u \quad \text{in } L_{\mathbb F}^2(0,T;L^2(\Gamma ;\mathbb R^{n_{-}})).\vspace{-2mm}
\end{equation}
For each $m$, a standard lifting argument provides $\widetilde{U}_m \in C_{\mathbb F}^1([0,T];L^2(\Omega;H^1(G;\mathbb{R}^N)))$ satisfying $\widetilde{U}_m|_{\Gamma}=U_m$ and $\widetilde{U}_m(0,x)=0$.

\textbf{Step 2. Reduction to a problem with homogeneous boundary data.}

Consider the auxiliary stochastic hyperbolic system with zero boundary condition:\vspace{-2mm}
\begin{equation} \label{zeroboundarydata}
\begin{cases}
\begin{aligned}
& d\tilde y_m+\sum_{i=1}^n A_i(x) \tilde y_{m,x_i}\,dt
= \big(B_1 \tilde y_m + B_3 v + F_m\big)dt  +\big[B_2(\tilde y_m+\widetilde{U}_m)+ v\big]dW(t)
&& \text{in } Q_T,\\[4pt]
&(\tilde \zeta_m)_{-}=0 && \text{on } \Sigma_T,\\[2pt]
&\tilde y_m(0)=y_0 && \text{in } G,
\end{aligned}
\end{cases}\vspace{-2mm}
\end{equation}
where $F_m = B_1\widetilde{U}_m - \widetilde{U}_{m,t} - \sum_{i=1}^n A_i \widetilde{U}_{m,x_i}$ and
$\tilde{\zeta}_m=((\tilde{\zeta}_m)_{+},(\tilde{\zeta}_m)_0,(\tilde{\zeta}_m)_{-})^{\top}\deq\Pi^{-1}\tilde{y}_{m}$.
By Lemma \ref{C0semigroup} and \cite[Theorem 3.14]{LZ21}, system \eqref{zeroboundarydata} admits a unique mild solution $\tilde y_m \in C_{\mathbb F}([0,T];L^2(\Omega;L^2(G;\mathbb{R}^N)))$.

Set $y_m \deq \tilde y_m + \widetilde{U}_m$. On the boundary,\vspace{-2mm}
$$
(\zeta_m)_{-} = \big[\Pi^{-1}(\tilde y_m+\widetilde{U}_m)\big]_{-}
= (\tilde \zeta_m)_{-} + u_m = u_m.\vspace{-2mm}
$$
Now fix $\tau \in (0,T]$ and $z_\tau \in L_{\mathcal F_\tau}^2(\Omega;L^2(G;\mathbb{R}^N))$. Applying It\^o's formula and integration by parts to the pair $(y_m,z)$ (where $(z,Z)$ solves the adjoint system \eqref{eqAdjoint}) gives\vspace{-2mm}
\begin{equation} \label{m1m2}
\mathbb{E}\int_G \langle y_m(\tau),z_\tau \rangle dx
- \int_G \langle y_0,z(0) \rangle dx
= \mathbb{E}\int_0^\tau \int_G \langle v, Z \rangle dx\,dt
- \mathbb{E}\int_0^\tau \int_{\Gamma} \langle \xi_{-}, \Lambda_{-} u_m \rangle d\Gamma\,dt.\vspace{-2mm}
\end{equation}

\textbf{Step 3. Energy estimates and passage to the limit.}

For any $m_1,m_2\in\mathbb N$, \eqref{m1m2} implies\vspace{-2mm}
$$
\mathbb{E}\int_G \langle y_{m_1}(\tau)-y_{m_2}(\tau),\,z_\tau \rangle dx
= -\mathbb{E}\int_0^\tau \int_{\Gamma} \langle \xi_{-},\, \Lambda_{-}(u_{m_1}-u_{m_2}) \rangle d\Gamma\,dt.\vspace{-2mm}
$$
Choose $z_\tau$ with $\|z_\tau\|_{L_{\mathcal F_\tau}^2(\Omega;L^2(G;\mathbb R^N))}=1$ such that\vspace{-2mm}
$$
\mathbb{E}\int_G \langle y_{m_1}(\tau)-y_{m_2}(\tau),z_\tau \rangle dx
\ge \frac12 \| y_{m_1}(\tau)-y_{m_2}(\tau)\|_{L_{\mathcal F_\tau}^2(\Omega;L^2(G;\mathbb{R}^N))}.\vspace{-2mm}
$$
Then, using Proposition \ref{hidden}, we have that\vspace{-2mm}
$$
\begin{aligned}
&\| y_{m_1}(\tau)-y_{m_2}(\tau)\|_{L_{\mathcal F_\tau}^2(\Omega;L^2(G;\mathbb R^N))} \\
& \le 2\Big|\mathbb{E}\int_0^\tau\int_{\Gamma} \langle \xi_{-},\,\Lambda_{-}(u_{m_1}-u_{m_2}) \rangle d\Gamma\,dt\Big| \\
& \le 2\| u_{m_1}-u_{m_2}\|_{L_{\mathbb{F}}^2(0,T;L^2(\Gamma ;\mathbb R^{n_{-}}))}
\,\|\Lambda_{-}\xi_{-}\|_{L_{\mathbb{F}}^2(0,T;L^2(\Gamma ;\mathbb R^{n_{-}}))} \\
& \le 2C\| u_{m_1}-u_{m_2}\|_{L_{\mathbb{F}}^2(0,T;L^2(\Gamma ;\mathbb R^{n_{-}}))}
\,\|z_\tau\|_{L_{\mathcal{F}_\tau}^2(\Omega;L^2(G;\mathbb{R}^N))} \\
& = 2C\| u_{m_1}-u_{m_2}\|_{L_{\mathbb{F}}^2(0,T;L^2(\Gamma ;\mathbb{R}^{n_{-}}))},
\end{aligned}\vspace{-2mm}
$$
where the constant $C$ is independent of $\tau$. Consequently,\vspace{-2mm}
$$
\|y_{m_1}-y_{m_2}\|_{C_{\mathbb{F}}([0,T];L^2(\Omega;L^2(G;\mathbb{R}^N)))}
\le C\| u_{m_1}-u_{m_2}\|_{L_{\mathbb{F}}^2(0,T;L^2(\Gamma ;\mathbb{R}^{n_{-}}))}.\vspace{-2mm}
$$
Hence $\{y_m\}_{m=1}^\infty$ is a Cauchy sequence in $C_{\mathbb{F}}([0,T];L^2(\Omega;L^2(G;\mathbb R^N)))$. Denote its limit by $y$. Letting $m\to\infty$ in \eqref{m1m2} shows that $y$ satisfies the transposition identity \eqref{deftrans}; therefore $y$ is a transposition solution of \eqref{fq}.

Finally, choose $z_\tau$ with $\|z_\tau\|_{L_{\mathcal F_\tau}^2(\Omega;L^2(G;\mathbb R^N))}=1$ such that\vspace{-2mm}
$$
\mathbb{E}\int_G\langle y(\tau),z_\tau\rangle dx
\ge \frac12 \| y(\tau)\|_{L_{\mathcal F_\tau}^2(\Omega;L^2(G;\mathbb R^N))}.\vspace{-2mm}
$$
Combining \eqref{deftrans} with Proposition \ref{hidden} yields\vspace{-2mm}
\begin{align*}
&\| y(\tau)\|_{L_{\mathcal F_\tau}^2(\Omega;L^2(G;\mathbb R^N))} \\
&\le 2\Big( \Big|\int_G \langle y_0,z(0) \rangle dx\Big|
+\Big|\mathbb{E}\int_0^\tau \int_G \langle v, Z \rangle dx\,dt\Big|
+\Big|\mathbb{E}\int_0^\tau \int_{\Gamma} \langle \xi_{-}, \Lambda_{-} u \rangle d\Gamma\,dt\Big| \Big) \\
&\le C\Big( \| y_0\|_{L^2(G;\mathbb R^N)}
+\| v\|_{L^2_{\mathbb F}(0,T;L^2(G;\mathbb R^N))}
+\| u\|_{L_{\mathbb F}^2(0,T;L^2(\Gamma ;\mathbb R^{n_{-}}))}\Big),\vspace{-2mm}
\end{align*}
with a constant $C>0$ independent of $\tau$. This completes the proof.
\end{proof}

\section{Proof of the observability estimate and the exact controllability} \label{S4}

In this section, we prove our main exact controllability result \cref{thmControllability}. By the classical duality argument, we only need to prove the following observability estimate for the system \eqref{eqAdjoint}.

\begin{proposition}
\label{thmObservability}
Assuming \cref{cond1} holds and $ T > T_{0} $, then
there exists a constant $C>0$ such that for any $ z_{T} \in L^{2}_{\mathcal{F}_{T}}(\Omega; L^{2}(G ; \mathbb{R}^{N})) $, it holds that\vspace{-2mm}
\begin{align}
\label{eqObservability}
\|z_{T}\|_{L^{2}_{\mathcal{F}_{T}}(\Omega; L^{2}(G ; \mathbb{R}^{N}))}^{2}
\leq
C\big(
\|\xi_{-}\|_{L^{2}_{\mathbb{F}}(0,T; L^{2}(\Gamma  ; \mathbb{R}^{n_{-}}))}^{2}
+ \|Z\|^{2}_{L^{2}_{\mathbb{F}}(0,T; L^{2}(G ; \mathbb{R}^{N}))}
\big),
\end{align}\vspace{-2mm}
where $ (z,Z) $ is the solution to \eqref{eqAdjoint}.
\end{proposition}

In order to prove \cref{thmObservability}, we first establish the following fundamental weighted identity.

\begin{lemma}
Let $ h $ be an $ H^{1}(G;\mathbb{R}^{N}) $-valued It\^{o} process.
Set $ \varphi \in C^{1}(\mathbb{R} \times \mathbb{R}^{n}) $, $ \lambda > 0 $, $ \theta = e^{\lambda \varphi} $, and $ w = \theta h $.
Then, for any $ t \in [0,T] $ and a.e. $ (x, \omega) \in G \times \Omega $, it holds that\vspace{-2mm}
\begin{equation}
\begin{aligned}
\label{eqWeightedIdentity}
&2 \lambda \theta \Big \langle w,  \Big(d h + \sum_{i=1}^{n} A_{i} h_{x_{i}} d t\Big) \Big \rangle\\
& =
\lambda \Big[\sum_{i=1}^{n} \partial_{x_{i}} \langle w,  A_{i} w \rangle d t
+ d \langle w,  w \rangle
- \langle d w,  d w \rangle
- \sum_{i=1}^{n} \langle w ,  A_{i,x_{i}} w \rangle d t \Big]\\ 
&\q
- 2 \lambda^{2} \Big\langle w,  \Big( \sum_{i=1}^{n} A_{i} \varphi_{x_{i}} + \varphi_{t} \Big) w \Big\rangle d t.
\end{aligned}\vspace{-2mm}
\end{equation}
\end{lemma}
\begin{proof}

By the definition of $ w $, we have\vspace{-2mm}
\begin{equation*}
\begin{aligned}
\theta \Big(d h + \sum_{i=1}^{n} A_{i} h_{x_{i}} d t\Big)
& =
d w
- \lambda \varphi_{t} w d t
- \lambda \sum_{i=1}^{n} A_{i} \varphi_{x_{i}} w d t
+ \sum_{i=1}^{n} A_{i} w_{x_{i}} d t.
\end{aligned}\vspace{-2mm}
\end{equation*}
Thanks to  It\^o's formula, it holds that\vspace{-2mm}
\begin{equation*}
\begin{aligned}
\langle w,  d w \rangle
&=
\frac{1}{2} d \langle w,  w \rangle
- \frac{1}{2} \langle d w,  d w \rangle.
\end{aligned}\vspace{-2mm}
\end{equation*}
For each $i=1,\dots,n$, we obtain that\vspace{-2mm}
\begin{equation*}
\begin{aligned}
\Big \langle w,  \sum_{i=1}^{n} A_{i} w_{x_{i}} \Big \rangle d t
&=
\frac{1}{2} \sum_{i=1}^{n} \partial_{x_{i}} \langle w,  A_{i} w \rangle d t
- \frac{1}{2} \sum_{i=1}^{n} \langle w ,  A_{i,x_{i}} w \rangle d t.
\end{aligned}\vspace{-2mm}
\end{equation*}
Combining the above identities yields the desired result.
\end{proof}

\begin{proof}[Proof of  \cref{thmObservability}]
We split the proof into three steps.

\textbf{Step 1. Derivation of a Carleman estimate.}

Take $h = z$ in the weighted identity \eqref{eqWeightedIdentity}.
Set $\varphi(t,x) = \beta t + \eta(x)$, where $\beta > 0$ will be fixed later and $\eta \in C^{1}(\overline{G})$ is the function from Condition \ref{cond1}.
Integrating \eqref{eqWeightedIdentity} over $Q_{T} = (0,T) \times G$, using integration by parts, taking expectation, and substituting the adjoint system \eqref{eqAdjoint}, we obtain\vspace{-2mm}
\begin{equation}
\begin{aligned}
\label{eqProfOb1}
&2 \lambda \,\mathbb{E} \int_{0}^{T} \int_{G} \theta
\Big\langle w,\; \Big(dz + \sum_{i=1}^{n} A_{i} z_{x_{i}} \,dt\Big)\Big\rangle dx  \\
&= \lambda \,\mathbb{E} \int_{0}^{T} \int_{\Gamma }
\sum_{i=1}^{n}\big\langle w,\; A_{i}\nu_{i} w\big\rangle \,d\Gamma \,dt
+ \lambda \,\mathbb{E}\!\int_{G} |w(T)|^{2} \,dx
- \lambda \int_{G} |w(0)|^{2} \,dx  \\
&\quad- \lambda \,\mathbb{E}\!\int_{0}^{T} \int_{G} \theta^{2} |Z - B_3^{\top} z|^{2} \,dx\,dt
- \lambda \,\mathbb{E} \int_{0}^{T} \int_{G}
\sum_{i=1}^{n}\big\langle w,\;  A_{i,x_{i}}  w\big\rangle \,dx\,dt  \\
&\quad- 2\lambda^{2} \,\mathbb{E} \int_{0}^{T} \int_{G}
\Big\langle w,\; \Big( \sum_{i=1}^{n} A_{i}\varphi_{x_{i}} + \varphi_{t} \Big) w\Big\rangle \,dx\,dt .
\end{aligned}\vspace{-2mm}
\end{equation}
By the boundary condition $\xi_{+}=0$ and the definition $\xi = \Pi^{-1}z$,\vspace{-2mm}
\begin{equation}
\label{eqProfOb2}
\lambda \,\mathbb{E}\!\int_{0}^{T}\!\!\int_{\Gamma }
\sum_{i=1}^{n}\big\langle w,\; A_{i}\nu_{i} w\big\rangle \,d\Gamma dt
= \lambda \,\mathbb{E}\!\int_{0}^{T}\!\!\int_{\Gamma } \theta^{2}
\langle \xi_{-},\; \Lambda_{-} \xi_{-} \rangle \,d\Gamma dt
\ge - C\lambda \,\mathbb{E}\!\int_{0}^{T}\!\!\int_{\Gamma } \theta^{2} |\xi_{-}|^{2} \,d\Gamma dt.\vspace{-2mm}
\end{equation}
Choose $\beta < c_{0}$ with $c_{0}$ given in Condition \ref{cond1}. Then, because $\varphi_{x_i} = \eta_{x_i}$ and $\varphi_{t} = \beta$, we have\vspace{-2mm}
\begin{equation}
\begin{aligned}
\label{eqProfOb3}
&- 2\lambda^{2} \,\mathbb{E} \int_{0}^{T} \int_{G}
\Big\langle w,\; \Big( \sum_{i=1}^{n} A_{i}\varphi_{x_{i}} + \varphi_{t} \Big) w\Big\rangle \,dx\,dt \\
& = - 2\lambda^{2} \,\mathbb{E} \int_{0}^{T} \int_{G} \theta^{2}
\Big\langle z,\; \Big( \sum_{i=1}^{n} A_{i}\eta_{x_{i}} + \beta \Big) z\Big\rangle \,dx\,dt \\
& \ge 2(c_{0}-\beta)\,\lambda^{2} \,\mathbb{E}\!\int_{0}^{T}\!\!\int_{G} \theta^{2} |z|^{2} \,dx\,dt .
\end{aligned}\vspace{-2mm}
\end{equation}
Moreover,\vspace{-2mm}
\begin{equation}
\label{eqProfOb7}
- \lambda \,\mathbb{E}\!\int_{0}^{T}\!\!\int_{G}
\sum_{i=1}^{n}\big\langle w,\; A_{i,x_{i}} w\big\rangle \,dx\,dt
\ge - C\lambda \,\mathbb{E}\!\int_{0}^{T}\!\!\int_{G} \theta^{2} |z|^{2} \,dx\,dt .\vspace{-2mm}
\end{equation}
Using Cauchy--Schwarz inequality,\vspace{-2mm}
\begin{equation}
\label{eqProfOb111}
- \lambda \,\mathbb{E}\!\int_{0}^{T}\!\!\int_{G} \theta^{2}|Z-B_3^{\top}z|^{2} \,dx\,dt
\ge -2\lambda \,\mathbb{E}\!\int_{0}^{T}\!\!\int_{G} \theta^{2}|Z|^{2} \,dx\,dt
- C\lambda \,\mathbb{E}\!\int_{0}^{T}\!\!\int_{G} \theta^{2}|z|^{2} \,dx\,dt .\vspace{-2mm}
\end{equation}
From \eqref{eqAdjoint} and Cauchy--Schwarz inequality,\vspace{-2mm}
\begin{equation}
\label{eqProfOb4}
2\lambda \,\mathbb{E}\!\int_{0}^{T}\!\!\int_{G} \theta
\Big\langle w,\; \Big(dz + \sum_{i=1}^{n} A_{i}z_{x_{i}}dt\Big)\Big\rangle dx
\le C\lambda \,\mathbb{E}\!\int_{0}^{T}\!\!\int_{G} \theta^{2}|z|^{2} \,dx\,dt
+ C\lambda \,\mathbb{E}\!\int_{0}^{T}\!\!\int_{G} \theta^{2}|Z|^{2} \,dx\,dt .\vspace{-2mm}
\end{equation}
Collecting \eqref{eqProfOb1}--\eqref{eqProfOb4}, and noting that the term \eqref{eqProfOb3} is quadratic in $\lambda$, we know that there exists a constant $\widehat{\lambda}_{1}>0$ such that for all $\lambda \ge \widehat{\lambda}_{1}$,\vspace{-2mm}
\begin{equation}
\begin{aligned}
\label{eqProfOb5}
& C\lambda \,\mathbb{E}\!\int_{0}^{T}\!\!\int_{G} \theta^{2} |Z|^{2} \,dx\,dt
+ C\lambda \,\mathbb{E}\!\int_{0}^{T}\!\!\int_{\Gamma } \theta^{2} |\xi_{-}|^{2} \,d\Gamma dt
+ \lambda \int_{G} \theta(0)^{2}|z(0)|^{2} \,dx \\
&\ge \lambda \,\mathbb{E}\!\int_{G} \theta(T)^{2} |z_{T}|^{2} \,dx
+ \lambda^{2} \,\mathbb{E}\!\int_{0}^{T}\!\!\int_{G} \theta^{2} |z|^{2} \,dx\,dt .
\end{aligned}\vspace{-2mm}
\end{equation}

\textbf{Step 2. Estimating $\displaystyle\int_{G} \theta(0)^{2}|z(0)|^{2} dx$.}

Applying It\^o's formula to $|z|^2$ on $[t_{1},t_{2}] \subset [0,T]$ and using \eqref{eqAdjoint},\vspace{-2mm}
\begin{align*}
&\mathbb{E} \int_{G} |z(t_{2})|^{2} dx - \mathbb{E} \int_{G} |z(t_{1})|^{2} dx \\
&\ge -2\mathbb{E} \int_{t_{1}}^{t_{2}} \int_{\Gamma }
\langle \xi_{-},\; \Lambda_{-}\xi_{-} \rangle \,d\Gamma dt
- C\,\mathbb{E}\!\int_{t_{1}}^{t_{2}} \int_{G} |z|^{2} dx\,dt
+ \frac12 \,\mathbb{E} \int_{t_{1}}^{t_{2}} \int_{G} |Z|^{2} dx\,dt \\
&\ge - C\,\mathbb{E}\!\int_{t_{1}}^{t_{2}} \int_{G} |z|^{2} dx\,dt .
\end{align*}\vspace{-2mm}
Gronwall's inequality then gives\vspace{-2mm}
$$
\int_{G} |z(0)|^{2} dx \le e^{CT} \,\mathbb{E}\!\int_{G} |z_{T}|^{2} dx .\vspace{-2mm}
$$
Recalling $\theta = e^{\lambda\varphi} = e^{\lambda(\beta t+\eta(x))}$, we obtain\vspace{-2mm}
$$
\begin{aligned}
\int_{G} \theta(0)^{2}|z(0)|^{2} dx
&= \int_{G} e^{2\lambda\eta(x)} |z(0)|^{2} dx \le e^{2\lambda \max_{x\in \overline G}\eta(x)} \int_{G} |z(0)|^{2} dx \\
&\le e^{CT + 2\lambda \max_{x\in \overline G}\eta(x)} \,\mathbb{E}\!\int_{G} |z_{T}|^{2} dx \\
&\le e^{CT - 2\lambda\beta T + 2\lambda\big(\max_{x\in \overline G}\eta(x)-\min_{x\in \overline G}\eta(x)\big)}
\; \mathbb{E}\!\int_{G} \theta(T)^{2} |z_{T}|^{2} dx .
\end{aligned}\vspace{-2mm}
$$

\textbf{Step 3. Completion of the proof.}

For $T > T_{0}$ (where $T_{0}$ is given in \eqref{eqMinimumTime}), choose $\beta = \frac{2c_{0}T_{0}}{T+T_{0}} < c_{0}$,
so that\vspace{-2mm}
$$
T > \frac{1}{\beta}\Big(\max_{x\in\overline{G}}\eta(x)-\min_{x\in\overline{G}}\eta(x)\Big)
= \frac{T+T_{0}}{2} > T_{0}.\vspace{-2mm}
$$
Then\vspace{-2mm}
$$
-\beta T + \big(\max_{x\in \overline G}\eta(x)-\min_{x\in \overline G}\eta(x)\big)
= \beta\bigg[ \frac{1}{\beta}\big(\max_{x\in \overline G}\eta(x)-\min_{x\in \overline G}\eta(x)\big) - T \bigg]
= \frac{\beta(T_{0}-T)}{2} < 0.\vspace{-2mm}
$$
Consequently, there exists $\widehat{\lambda}_{2}>0$ such that for all $\lambda \ge \widehat{\lambda}_{2}$,\vspace{-2mm}
\begin{equation}
\label{eqProfOb6}
\int_{G} \theta(0)^{2}|z(0)|^{2} dx \le
\frac12 \,\mathbb{E}\!\int_{G} \theta(T)^{2} |z_{T}|^{2} dx.\vspace{-2mm}
\end{equation}
Finally, set $\widehat{\lambda}_{3} = \max\{\widehat{\lambda}_{1},\widehat{\lambda}_{2}\}$.
For every $\lambda \ge \widehat{\lambda}_{3}$ and $T > T_{0}$, combining \eqref{eqProfOb5} and \eqref{eqProfOb6} yields the observability estimate \eqref{eqObservability}. This completes the proof.
\end{proof}

Finally, we conclude this section by proving Theorem \ref{thmControllability} via the observability estimate established in Theorem \ref{thmObservability}. The argument is a standard duality method; we include it here for completeness.
\begin{proof}[Proof of Theorem \ref{thmControllability}]
Fix initial data $y_{0} \in L^{2}(G; \mathbb{R}^{N})$ and a target state $y_{1} \in L^{2}_{\mathcal{F}_{T}}(\Omega; L^{2}(G; \mathbb{R}^{N}))$.
Define a linear subspace $\mathcal{X}$ of $L^{2}_{\mathbb{F}}(0,T; L^{2}(\Gamma ; \mathbb{R}^{n_{-}})) \times L^{2}_{\mathbb{F}}(0,T; L^{2}(G; \mathbb{R}^{N}))$ by\vspace{-2mm}
$$
\mathcal{X}=
\Big\{
\big(-\Lambda_{-} \xi_{-},\; Z\big)
\;\Big|\;
(z,Z) \text{ solves \eqref{eqAdjoint} with some }
z_{T} \in L^{2}_{\mathcal{F}_{T}}(\Omega; L^{2}(G ; \mathbb{R}^{N}))
\Big\},\vspace{-2mm}
$$
and introduce a linear functional $\mathcal{L}$ on $\mathcal{X}$ by\vspace{-2mm}
$$
\mathcal{L}\big(-\Lambda_{-} \xi_{-},\; Z\big)
= \mathbb{E}\!\int_{G} \langle y_{1},\; z_{T} \rangle\,dx
- \int_{G} \langle y_{0},\; z(0) \rangle\,dx.\vspace{-2mm}
$$
Theorem \ref{thmObservability} guarantees that $\mathcal{L}$ is bounded on $\mathcal{X}$.
By the Hahn--Banach theorem, $\mathcal{L}$ can be extended to a bounded linear functional on the whole space $L^{2}_{\mathbb{F}}(0,T; L^{2}(\Gamma ; \mathbb{R}^{n_{-}})) \times L^{2}_{\mathbb{F}}(0,T; L^{2}(G ; \mathbb{R}^{N}))$; we still denote the extension by $\mathcal{L}$.
Applying the Riesz representation theorem, we obtain a pair\vspace{-2mm}
$$
(u,v) \in L^{2}_{\mathbb{F}}(0,T; L^{2}(\Gamma ; \mathbb{R}^{n_{-}})) \times L^{2}_{\mathbb{F}}(0,T; L^{2}(G; \mathbb{R}^{N}))\vspace{-2mm}
$$
such that\vspace{-2mm}
\begin{equation}
\label{eqProofControllability1}
\mathbb{E}\!\int_{G} \langle y_{1},\; z_{T} \rangle\,dx
- \int_{G} \langle y_{0},\; z(0) \rangle\,dx
= - \mathbb{E}\!\int_{0}^{T}\!\!\int_{\Gamma } \langle u,\; \Lambda_{-} \xi_{-} \rangle\,d\Gamma\,dt
+ \mathbb{E}\!\int_{0}^{T}\!\!\int_{G} \langle v,\; Z \rangle\,dx\,dt .\vspace{-2mm}
\end{equation}

Let $y$ be the transposition solution of \eqref{fq} corresponding to the controls $(u,v)$. According to Definition \ref{defDefinitionTrans},\vspace{-2mm}
\begin{equation}
\label{eqProofControllability2}
\mathbb{E}\!\int_{G} \langle y(T),\; z_{T} \rangle\,dx
- \int_{G} \langle y_{0},\; z(0) \rangle\,dx
= - \mathbb{E}\!\int_{0}^{T}\!\!\int_{\Gamma } \langle u,\; \Lambda_{-} \xi_{-} \rangle\,d\Gamma\,dt
+ \mathbb{E}\!\int_{0}^{T}\!\!\int_{G} \langle v,\; Z \rangle\,dx\,dt .\vspace{-2mm}
\end{equation}

Subtracting \eqref{eqProofControllability1} from \eqref{eqProofControllability2} gives\vspace{-2mm}
$$
\mathbb{E}\!\int_{G} \langle y(T) - y_{1},\; z_{T} \rangle\,dx = 0
\qquad
\text{for all } z_{T} \in L^{2}_{\mathcal{F}_{T}}(\Omega; L^{2}(G; \mathbb{R}^{N})),\vspace{-2mm}
$$
which implies $y(T) = y_{1}$, $\mathbb{P}$-a.s. This completes the proof.
\end{proof}

\section{The lack of the exact controllability} \label{S5}

In this section, we prove Theorems \ref{thm1.2}--\ref{thm1.4}.

\begin{proof}[Proof of Theorem \ref{thm1.2}]
In the case of $B_3=0$ and $u\equiv0$, system \eqref{fq} reads as\vspace{-2mm}
  \begin{equation}\label{5.1}
    \begin{cases}
      dy + \sum_{i=1}^{n} A_i  y_{x_i}\, dt
= B_1 y\, dt + (B_2 y + v)\, dW(t) & \text{in } Q_T,\\
\zeta_{-} = 0 & \text{on } \Sigma_T,\\
y(0) = y_0 & \text{in } G.
    \end{cases}\vspace{-2mm}
  \end{equation}
We only need to prove that the attainable set $\mathbb{A}_{T}$ of system \eqref{5.1} at time $T$ with initial data $y_0=0$ is not $L_{\mathcal{F}_{T}}^2(\Omega;L^2(G;\mathbb{R}^N))$.
  For $y_0=0$ and any $t\in[0,T]$, the solution to \eqref{5.1} satisfies that\vspace{-2mm}
  \begin{equation}\label{5.2}
    y(t)=\int_{0}^{t}S(t-s)B_1y \,ds +\int_{0}^{t}S(t-s)(B_2y+v)\,dW(s),\vspace{-2mm}
  \end{equation}
where $\{S(t)\}_{t\in[0,T]}$ is the $C_0$-semigroup generated by $\cA$. Taking mathematical expectation on both sides of \eqref{5.2}, we obtain that\vspace{-2mm}
  \begin{equation*}
    \dbE(y(t))=\dbE\int_{0}^{t}S(t-s)B_1y \,ds.\vspace{-2mm}
  \end{equation*}
Taking $L^2(G;\mathbb{R}^N)$ norm on both sides and noting that $B_1$ is deterministic, we conclude that
  \begin{equation}\label{5.4}
    \begin{aligned}
      \|\dbE(y(t))\|_{L^2(G;\mathbb{R}^{N})}
      &=\Big\|\int_0^{t}S(t-s)B_1\mathbb{E}(y)\,ds\Big\|_{L^2(G;\mathbb{R}^{N})}\\
      &\leq C_1\int_{0}^{t}e^{C_2(t-s)}\|\mathbb{E}(y(s))\|_{L^2(G;\mathbb{R}^{N})}ds,
    \end{aligned}
  \end{equation}
  where $C_1,C_2>0$ are constants independent on $t$. Denote by
  $$
  Y(t)\deq e^{-C_2t}\|\dbE(y(t))\|_{L^2(G;\mathbb{R}^{N})},
  $$
  then \eqref{5.4} can be rewritten as
  $$
  Y(t)\leq C_1\int_{0}^{t}Y(s)ds.
  $$
Using Gronwall's inequality and noting that $Y(0)=0$, we finally obtain that $Y(t)=0$ for any $t\in[0,T]$, which implies that $\mathbb{E}(y(t))=0$ for any $t\in[0,T]$. Thus, for $y_0=0$, if we choose $y_1\in L^2_{\mathcal{F}_{T}}(\Omega;L^2(G;\mathbb{R}^{N}))$ such that $\dbE(y_1)\neq0$, then $y_1$ is not in the attainable set $\mathbb{A}_T$ of system \eqref{5.1} at time $T$, which completes the proof.
\end{proof}

\begin{remark}
Naturally, we should consider whether the conclusion is valid after removing the assumptions that $B_1$ is deterministic and $B_3=0$. However, we still do not know how to do this, and we remain this to be a future work.
\end{remark}

In order to prove Theorems \ref{thm1.3} and \ref{thm1.4}, we first  recall a useful result.

\begin{lemma}[{\cite[Proposition 6.2]{LZ21}}]
\label{lemmaImpossibleRepresentation}
There exists a random variable $\varrho  \in L_{\mathcal{F}_T}^2(\Omega)$ such that it is impossible to find $(\varsigma_1, \varsigma_2) \in L_{\mathbb{F}}^2(0, T) \times C_{\mathbb{F}}([0, T] ; L^2(\Omega))$ and $\varsigma_0 \in \mathbb{R}$ satisfying that
\begin{equation*}
\varrho =\varsigma_0 +\int_0^T \varsigma_1(t) \mathrm{d} t+\int_0^T \varsigma_2(t) \mathrm{d} W(t).
\end{equation*}
\end{lemma}

\medskip


\begin{proof}[Proof of Theorem \ref{thm1.3}]
Let $\chi_1 \in C_0^{\infty}(G \setminus G_0 ; \mathbb{R}^N)$, where $G_0$ is the open subset introduced in Theorem \ref{thm1.3}, 
be a fixed function with $\|\chi_1\|_{L^2(G ; \mathbb{R}^N)} = 1$.
Assume, for contradiction, that system \eqref{fq} is exactly controllable at time $T$. Then for any initial state $y_0 \in L^2(G ; \mathbb{R}^N)$ there exist controls  $
(u,v) \in L_{\mathbb{F}}^2(0,T; L^2(\Gamma ; \mathbb{R}^{n_{-}})) \times L_{\mathbb{F}}^2(0,T; L^2(G ; \mathbb{R}^N))
$
with $\operatorname{supp} v \subset G_0$ such that the corresponding solution of \eqref{fq} satisfies $y(T) = \chi_1 \varrho$, where $\varrho$ is the random variable given in Lemma \ref{lemmaImpossibleRepresentation}.

Recall the approximating sequence $\{y_m\}_{m=1}^{\infty}$ constructed in the proof of Proposition \ref{wellposednessfq}, where $y_m = \tilde{y}_m + \widetilde{U}_m$.
Since $\tilde{y}_m$ is the mild solution of \eqref{zeroboundarydata}, it is also a weak solution. Integrating by parts in \eqref{zeroboundarydata} and using the equation for $\widetilde{U}_m$, we obtain\vspace{-2mm}
\begin{align*}
&\int_{G}\big\langle \tilde{y}_{m}(T),\;\chi_1\big\rangle dx - \int_{G}\big\langle y_0,\;\chi_1\big\rangle dx \\
&= \int_0^T \!\int_{G} \Big\langle -\sum_{i=1}^{n}A_i \tilde{y}_{m,x_i}+B_1\tilde{y}_m
-\sum_{i=1}^{n}A_i\widetilde{U}_{m,x_i}+B_1\widetilde{U}_m-\widetilde{U}_{m,t},\; \chi_1\Big\rangle dx\,dt \\
&\quad + \int_0^T\!\int_{G} \big\langle B_2(\tilde{y}_m+\widetilde{U}_m),\;\chi_1\big\rangle dx\,dW(t) \\
&= \int_{0}^{T}\!\!\int_{G} \Big\langle y_m,\; \sum_{i=1}^{n}(A_i\chi_1)_{x_i}+B_1^{\top}\chi_1\Big\rangle dx\,dt
+ \int_0^T\!\!\int_{G} \big\langle y_m,\; B_2^{\top}\chi_1\big\rangle dx\,dW(t) - \int_{G}\big\langle \widetilde{U}_m(T),\;\chi_1\big\rangle dx .
\end{align*}\vspace{-2mm}

Passing to the limit $m \to \infty$ yields\vspace{-2mm}
$$
\begin{aligned}
\int_G\big\langle y(T),\; \chi_1\big\rangle dx
&= \int_G\big\langle y_0,\; \chi_1\big\rangle dx
+ \int_0^T\!\!\int_{G} \Big\langle y,\; \sum_{i=1}^{n}(A_i\chi_1)_{x_i}+B_1^{\top}\chi_1\Big\rangle dx\,dt \\
&\quad + \int_0^T\!\!\int_{G} \big\langle y,\; B_2^{\top}\chi_1\big\rangle dx\,dW(t).
\end{aligned}\vspace{-2mm}
$$
Because $y(T) = \chi_1 \varrho$ and $\|\chi_1\|_{L^2(G;\mathbb{R}^N)}=1$, we deduce\vspace{-2mm}
\begin{equation}\label{eq5.6}
\begin{aligned}
\varrho
&= \int_G\big\langle \chi_1\varrho,\; \chi_1\big\rangle dx
= \int_G\big\langle y(T),\; \chi_1\big\rangle dx \\
&= \int_G\big\langle y_0,\; \chi_1\big\rangle dx
+ \int_0^T\!\!\int_{G} \Big\langle y,\; \sum_{i=1}^{n}(A_i\chi_1)_{x_i}+B_1^{\top}\chi_1\Big\rangle dx\,dt  + \int_0^T\!\!\int_{G} \big\langle y,\; B_2^{\top}\chi_1\big\rangle dx\,dW(t).
\end{aligned}\vspace{-2mm}
\end{equation}
From $y \in C_{\mathbb{F}}([0,T]; L^2(\Omega ; L^2(G ; \mathbb{R}^N)))$ and the assumption
$B_2 \in C_{\mathbb{F}}([0,T]; L^{\infty}(\Omega ; L^{\infty}(G ; \mathbb{R}^{N \times N})))$, it follows that\vspace{-2mm}
$$
\int_{G}\Big\langle y,\; \sum_{i=1}^{n}(A_i\chi_1)_{x_i}+B_1^{\top}\chi_1\Big\rangle dx \in L_{\mathbb{F}}^2(0,T),\vspace{-2mm}
$$
and\vspace{-2mm}
$$
\int_{G}\big\langle y,\; B_2^{\top}\chi_1\big\rangle dx \in C_{\mathbb{F}}([0,T]; L^2(\Omega)).\vspace{-2mm}
$$

These regularity properties make the right‑hand side of \eqref{eq5.6} a stochastic process that belongs to the class described in Lemma \ref{lemmaImpossibleRepresentation}. However, the random variable $\varrho$ itself does  not  admit such a representation, by the statement of that lemma. This contradiction shows that system \eqref{fq} cannot be exactly controllable when the distributed control $v$ is supported only in $G_0$. Hence the proof is complete.
\end{proof}
\begin{proof}[Proof of Theorem \ref{thm1.4}]
The argument parallels that of Theorem \ref{thm1.3}.
Pick $\chi_2 \in C_0^{\infty}(G ; \mathbb{R}^N)$ with $\|\chi_2\|_{L^2(G ; \mathbb{R}^N)} = 1$.
Assume, for a contradiction, that system \eqref{1.12} is exactly controllable at time $T$.
Then for every initial state $y_0 \in L^2(G; \mathbb{R}^N)$ there exist controls  $
(u_1, v_1) \in L_{\mathbb{F}}^2(0,T; L^2(\Gamma ; \mathbb{R}^{n_{-}})) \times L_{\mathbb{F}}^2(0,T; L^2(G ; \mathbb{R}^N))
$
such that the corresponding solution of \eqref{1.12} satisfies $y(T) = \chi_2 \varrho$, where $\varrho$ is the random variable given in Lemma \ref{lemmaImpossibleRepresentation}.

Proceeding as in the proof of Theorem \ref{thm1.3}, we obtain\vspace{-2mm}
\begin{equation} \label{eq5.7}
\begin{aligned}
\varrho
&= \int_G \big\langle \chi_2 \varrho,\; \chi_2 \big\rangle dx
= \int_G \big\langle y(T),\; \chi_2 \big\rangle dx \\
&= \int_G \big\langle y_0(x),\; \chi_2(x) \big\rangle dx
+ \int_0^T\!\!\int_G \Big[\,\Big\langle y,\; \sum_{i=1}^n (A_i \chi_2)_{x_i}+B_1^{\top} \chi_2\Big\rangle
+ \langle v_1,\; \chi_2 \rangle\Big] dx\,dt \\
&\quad + \int_0^T\!\!\int_G \big\langle y,\; B_2^{\top}\chi_2 \big\rangle dx\,dW(t).
\end{aligned}\vspace{-2mm}
\end{equation}
Because $y \in C_{\mathbb{F}}([0,T]; L^2(\Omega ; L^2(G ; \mathbb{R}^N)))$ and
$B_2 \in C_{\mathbb{F}}([0,T]; L^{\infty}(\Omega ; L^{\infty}(G ; \mathbb{R}^{N \times N})))$, we have\vspace{-2mm}
$$
\int_G \Big[\,\Big\langle y,\; \sum_{i=1}^n (A_i \chi_2)_{x_i}+B_1^{\top} \chi_2\Big\rangle
+ \langle v_1,\; \chi_2 \rangle\Big] dx \;\in\; L_{\mathbb{F}}^2(0,T),\vspace{-2mm}
$$
and\vspace{-2mm}
$$
\int_G \big\langle y,\; B_2^{\top}\chi_2 \big\rangle dx \;\in\; C_{\mathbb{F}}([0,T]; L^2(\Omega)).\vspace{-2mm}
$$
Consequently, the right‑hand side of \eqref{eq5.7} belongs to the class of stochastic processes described in Lemma \ref{lemmaImpossibleRepresentation}, whereas the random variable $\varrho$ itself does  not  admit such a representation. This contradiction shows that system \eqref{1.12} cannot be exactly controllable. The proof is complete.
\end{proof}

\section* {Acknowledgements}
 This paper was completed during the visit of Y. Wang and H. Yang to Prof. Zuazua at the FAU DCN–AvH Chair. They sincerely thank Prof.~Zuazua for his hospitality and for many fruitful discussions.

\begin{appendices}

\section{The proof of Lemma \ref{C0semigroup}} \label{appendixA}

The main goal of this section is to prove  Lemma \ref{C0semigroup}.

We first recall a key trace regularity fact.
\begin{lemma}[{\cite[Proposition 6.8]{CP11}}] \label{L2}
There exists a constant $C>0$ such that for every $\mathbf y \in L^2(G;\mathbb R^N)$ with $\mathcal A\mathbf y \in L^2(G;\mathbb R^N)$, the trace $
\Big(\sum_{i=1}^n \nu_i(x)A_i(x)\Big)\mathbf y \;\Big|_{\Gamma }
$ 
belongs to $L^{2}(\Gamma ;\mathbb R^N)$ and satisfies
$$
\Big\| \Big(\sum_{i=1}^n \nu_i(\cdot)A_i(\cdot)\Big)\mathbf y\Big\|_{L^{2}(\Gamma ;\mathbb R^N)}
\le C\bigl(\|\mathbf y\|_{L^2(G;\mathbb R^N)}+\|\mathcal A\mathbf y\|_{L^2(G;\mathbb R^N)}\bigr).
$$
\end{lemma}

For any $\mathbf y \in L^2(G;\mathbb R^N)$ with $\mathcal A\mathbf y \in L^2(G;\mathbb R^N)$, Lemma \ref{L2} guarantees that the boundary term is well defined. Using the diagonalization provided by $\Pi$, we obtain
$$
\begin{aligned}
\Pi^{-1}\Big(\sum_{i=1}^n \nu_i A_i\Big)\mathbf y\Big|_{\Gamma }
&= \Pi^{-1}\Big(\sum_{i=1}^n \nu_i A_i\Big)\Pi\;\Pi^{-1}\mathbf y\Big|_{\Gamma } \\
&= \begin{pmatrix}
\Lambda_+ & 0 & 0 \\
0 & 0 & 0 \\
0 & 0 & \Lambda_-
\end{pmatrix}
\begin{pmatrix}
\widetilde{\mathbf y}_+ \\[2pt]
\widetilde{\mathbf y}_0 \\[2pt]
\widetilde{\mathbf y}_-
\end{pmatrix}
= \begin{pmatrix}
\Lambda_+ \widetilde{\mathbf y}_+ \\[2pt]
0 \\[2pt]
\Lambda_- \widetilde{\mathbf y}_-
\end{pmatrix} \in L^2(\Gamma ;\mathbb R^N).
\end{aligned}
$$
Because $\Pi^{-1}(x)$ is bounded and the diagonal matrices $\Lambda_{\pm}^{-1}(x)$ are uniformly bounded, we conclude that both the incoming variable $\widetilde{\mathbf y}_{-}$ and the outgoing variable $\widetilde{\mathbf y}_{+}$ are bounded in the $L^{2}(\Gamma)$-norm.

\begin{lemma}
\label{lemmaAdjoint}
The adjoint operator of $\mathcal{A}$ is\vspace{-2mm}
$$
\mathcal{A}^*\Bz \deq \sum\limits_{i=1}^n A_i(x) \Bz_{x_i}+\sum_{i=1}^n A_{i,x_i}(x)\Bz,\vspace{-2mm}
$$
with the domain\vspace{-2mm}
$$
D(\mathcal{A}^*)=\{\Bz \in L^2( G;\mathbb R^N)\mid \mathcal{A}^*\Bz \in L^2( G;\mathbb R^N),~\wt{\Bz}_+=0\},\vspace{-2mm}
$$
where $\wt{\Bz}_+$ is given by the decomposition\vspace{-2mm}
$$
\wt{\Bz}=\begin{pmatrix}
\wt{\Bz}_+,
\wt{\Bz}_0,
\wt{\Bz}_-
\end{pmatrix}^\top=\Pi^{-1}\Bz.\vspace{-2mm}
$$
\end{lemma}

\begin{proof}
For $y \in D(\mathcal A)$ and $z \in D(\mathcal A^*)$, note that the boundary value $\wt{\By}_{+}$ and $\wt{\Bz}_{+}$ (resp. $\wt{\By}_{-}$ and $\wt{\Bz}_{-}$) are in the space $L^2(\Gamma ;\mathbb R^{n_{+}})$ (resp. $L^2(\Gamma ;\mathbb R^{n_{-}})$), hence the following integration by parts is legitimate.
We have\vspace{-2mm}
\begin{align*}
\langle \mathcal{A}\By,\Bz \rangle_{L^2( G;\mathbb R^N)}
=&
-\sum_{i=1}^n\int_G \langle \By_{x_i}, A_i(x) \Bz \rangle dx
\\
=&
-\sum_{i=1}^n\int_G \langle \By, (A_i(x) \Bz)_{x_i} \rangle dx+\int_G \Big \langle \By,  \Big [\sum_{i=1}^n A_i(x) \Bz_{x_i}+\sum_{i=1}^nA_{i,x_i}(x)\Bz \Big ] \Big \rangle dx
\\
=&-\int_{\Gamma }  \Big \langle \Pi^{-1}\By, \Pi^{-1}\Big (\sum_{i=1}^n \nu_i A_i\Big) \Pi (\Pi^{-1}\Bz) \Big \rangle d\Gamma +
\langle \By,\mathcal{A}^*\Bz \rangle_{L^2( G;\mathbb R^N)}
\\
=& -\int_{\Gamma } \begin{pmatrix}
\wt{\By}_+^\top & \wt{\By}_0^\top & \wt{\By}_-^\top
\end{pmatrix}\begin{pmatrix}
\Lambda_+(x) & 0 & 0 \\
0  & 0 & 0 \\
0 & 0 & \Lambda_-(x)
\end{pmatrix}\begin{pmatrix}
\wt{\Bz}_+ \\
\wt{\Bz}_0 \\
\wt{\Bz}_-
\end{pmatrix}
+\langle \By,\mathcal{A}^*\Bz \rangle_{L^2( G;\mathbb R^N)} \\
=&\langle \By,\mathcal{A}^*\Bz \rangle_{L^2( G;\mathbb R^N)},
\end{align*}\vspace{-2mm}
where we use the fact that $\Pi^{-1}=\Pi^{\top}$ in the third equality sign, and $\wt{\By}_-=0$ and $\wt{\Bz}_+=0$ in the last equality sign.
\end{proof}

\medskip

With the aforementioned preparations, we are now in a position to prove \Cref{C0semigroup}.

\begin{proof}[Proof of Lemma \ref{C0semigroup}]
We proceed in two steps.

\textbf{Step 1. Density and closedness of $\mathcal A$.}

The domain $D(\mathcal A)$ is dense in $L^2(G;\mathbb R^N)$ because it contains $C_0^\infty(G;\mathbb R^N)$.

To prove closedness, let $\{\mathbf y^k\}_{k=1}^{\infty}\subset D(\mathcal A)$ satisfy\vspace{-2mm}
$$
\mathbf y^k \xrightarrow{L^2} \mathbf y,\qquad
\mathcal A\mathbf y^k \xrightarrow{L^2} g,\vspace{-2mm}
$$
for some $\mathbf y,g\in L^2(G;\mathbb R^N)$.
In the sense of distributions, $\mathcal A\mathbf y^k\to\mathcal A\mathbf y$; hence the distributional limit of $\mathcal A\mathbf y^k$ coincides with $\mathcal A\mathbf y$. Since the limit is also $g$, we obtain $\mathcal A\mathbf y = g$ in $L^2(G;\mathbb R^N)$. Thus $\mathcal A\mathbf y\in L^2(G;\mathbb R^N)$ and $\mathcal A\mathbf y^k\to\mathcal A\mathbf y$ in $L^2(G;\mathbb R^N)$.

Now we apply Lemma \ref{L2} to the difference $\mathbf y^k-\mathbf y$:\vspace{-2mm}
$$
\begin{aligned}
	\| (\mathbf y^k)_{-} - \widetilde{\mathbf y}_{-}\|_{L^2(\Gamma ;\mathbb R^{n_{-}})}
	&\le C\Big\| \Big(\sum_{i=1}^n \nu_i(\cdot)A_i(\cdot)\Big)(\mathbf y^k-\mathbf y)\Big\|_{L^{2}(\Gamma ;\mathbb R^N)} \\
	&\le C\bigl(\|\mathbf y^k-\mathbf y\|_{L^2(G;\mathbb R^N)}+\|\mathcal A\mathbf y^k-\mathcal A\mathbf y\|_{L^2(G;\mathbb R^N)}\bigr)
	\to 0.
\end{aligned}\vspace{-2mm}
$$
Because $(\mathbf y^k)_{-}=0$ for each $k$, the limit gives $\widetilde{\mathbf y}_{-}=0$ in $L^2(\Gamma ;\mathbb R^{n_{-}})$. Hence $\mathbf y\in D(\mathcal A)$, and $\mathcal A$ is closed.

\textbf{Step 2. Quadratic estimates for $\mathcal A$ and its adjoint.}

For $\mathbf y\in D(\mathcal A)$, integration by parts yields\vspace{-2mm}
$$
\begin{aligned}
	\langle \mathcal A\mathbf y,\;\mathbf y\rangle_{L^2(G;\mathbb R^N)}
	&=-\sum_{i=1}^n \int_G \langle \mathbf y_{x_i},\; A_i(x)\mathbf y\rangle\,dx \\
	&=-\frac12\int_{\Gamma } \Big\langle \mathbf y,\; \Big(\sum_{i=1}^n \nu_i A_i\Big)\mathbf y\Big\rangle d\Gamma
	+\frac12\int_{G} \Big\langle \mathbf y,\; \Big(\sum_{i=1}^n A_{i,x_i}\Big)\mathbf y\Big\rangle dx .
\end{aligned}\vspace{-2mm}
$$
Using the decomposition $\widetilde{\mathbf y}=\Pi^{-1}\mathbf y$ and the boundary condition $\widetilde{\mathbf y}_{-}=0$,\vspace{-2mm}
$$
\Big(\sum_{i=1}^n \nu_i A_i\Big)\mathbf y\Big|_{\Gamma }
= \Pi\begin{pmatrix}\Lambda_+ \widetilde{\mathbf y}_+ \\ 0 \\ 0\end{pmatrix}.\vspace{-2mm}
$$
Therefore,\vspace{-2mm}
$$
\begin{aligned}
	\langle \mathcal A\mathbf y,\;\mathbf y\rangle_{L^2(G;\mathbb R^N)}
	&= -\frac12\int_{\Gamma } \langle \widetilde{\mathbf y}_+,\;\Lambda_+\widetilde{\mathbf y}_+\rangle d\Gamma
	+\frac12\int_{G} \Big\langle \mathbf y,\; \Big(\sum_{i=1}^n A_{i,x_i}\Big)\mathbf y\Big\rangle dx  \le M\|\mathbf y\|_{L^2(G;\mathbb R^N)}^2,
\end{aligned}\vspace{-2mm}
$$
with a constant $M>0$ depending only on the matrices $A_i$.

For the adjoint operator, let $\mathbf z\in D(\mathcal A^*)$. A similar computation gives\vspace{-2mm}
$$
\begin{aligned}
	\langle \mathcal A^*\mathbf z,\;\mathbf z\rangle_{L^2(G;\mathbb R^N)}
	&=\frac12\int_{\Gamma } \Big\langle \mathbf z,\; \Big(\sum_{i=1}^n \nu_i A_i\Big)\mathbf z\Big\rangle d\Gamma
	+\frac12\int_{G} \Big\langle \mathbf z,\; \Big(\sum_{i=1}^n A_{i,x_i}\Big)\mathbf z\Big\rangle dx .
\end{aligned}\vspace{-2mm}
$$
Now $\widetilde{\mathbf z}_{+}=0$ on $\Gamma$ (by the definition of $D(\mathcal A^*)$), so\vspace{-2mm}
$$
\Big(\sum_{i=1}^n \nu_i A_i\Big)\mathbf z\Big|_{\Gamma }
= \Pi\begin{pmatrix}0 \\ 0 \\ \Lambda_- \widetilde{\mathbf z}_-\end{pmatrix},\vspace{-2mm}
$$
and consequently\vspace{-2mm}
$$
\langle \mathcal A^*\mathbf z,\;\mathbf z\rangle_{L^2(G;\mathbb R^N)}
= \frac12\int_{\Gamma } \langle \widetilde{\mathbf z}_-,\;\Lambda_-\widetilde{\mathbf z}_-\rangle d\Gamma
+\frac12\int_{G} \Big\langle \mathbf z,\; \Big(\sum_{i=1}^n A_{i,x_i}\Big)\mathbf z\Big\rangle dx
\le M\|\mathbf z\|_{L^2(G;\mathbb R^N)}^2.\vspace{-2mm}
$$

\textbf{Conclusion.}
Steps 1 and 2 show that $\mathcal A$ is a densely defined, closed operator on $L^2(G;\mathbb R^N)$ and that both $\mathcal A$ and its adjoint $\mathcal A^*$ satisfy the dissipativity estimate\vspace{-2mm}
$$
\Re\,\langle \mathcal A u,\; u\rangle \le M\|u\|_{L^2(G;\mathbb R^N)}^2\;(\text{resp. } \Re\,\langle \mathcal A^* u,\; u\rangle \le M\|u\|_{L^2(G;\mathbb R^N)}^2),\qquad u\in D(\mathcal A)\;(\text{resp. } D(\mathcal A^*)).\vspace{-2mm}
$$
By the Lumer--Phillips theorem (see e.g., \cite[Corollary 2.2.3]{CZ12}), $\mathcal A$ generates a unique $C_0$-semigroup on $L^2(G;\mathbb R^N)$.
\end{proof}

\begin{remark} \label{RmkA1}
Since $L^2(G;\mathbb{R}^N)$ is a Hilbert space, it follows that $\mathcal A^*$ also generates a $C_0$-semigroup provided that $\mathcal A$ does.
\end{remark}
\end{appendices}

\end{document}